\documentclass[11pt, reqno]{amsart}
\usepackage{amssymb,latexsym,amsmath,amsfonts,mathdots,enumitem}
\usepackage{latexsym}
\usepackage[mathscr]{eucal}
\usepackage{colortbl,xcolor}
\usepackage{lmodern}
\usepackage{sansmathaccent}
\usepackage[latin1]{inputenc}
\usepackage{tikz}
\usepackage{physics}
\usetikzlibrary{shapes,arrows}
\usetikzlibrary{matrix,calc,shapes,arrows,positioning}
\pdfmapfile{+sansmathaccent.map}

\numberwithin{equation}{section}
\theoremstyle{plain}
\newtheorem{exam}{Example}

\newtheorem{thm}{Theorem}[section]

\newtheorem{cor}[thm]{Corollary}
\newtheorem{lem}[thm]{Lemma}
\newtheorem{prop}[thm]{Proposition}
\theoremstyle{definition}
\newtheorem{defn}[thm]{Definition}
\newtheorem{rem}[thm]{Remark}

\numberwithin{equation}{section}

\def\beq{\begin{eqnarray}}
	\def\eeq{\end{eqnarray}}
\def\beqa{\begin{eqnarray*}}
	\def\eeqa{\end{eqnarray*}}

\def\beqn{\begin{equation}}
	\def\eeqn{\end{equation}}

\def\mg#1{}

\renewcommand{\epsilon}{\varepsilon}
\renewcommand{\phi}{\varphi}

\renewcommand{\bf}[1]{\textbf{#1}}
\renewcommand{\it}[1]{\textit{#1}}
\renewcommand{\sc}[1]{\textsc{#1}}
\renewcommand{\sf}[1]{\textsf{#1}}

\DeclareMathOperator{\id}{\sf{id}}

\numberwithin{equation}{section}
\allowdisplaybreaks[4] 

\setlist[enumerate]{font=\upshape,noitemsep, topsep=0pt} 
\setlist[itemize]{noitemsep, topsep=0pt}

\begin{document}
	
	\title{Rational $\mathbf{\Theta_n}$-Inner Function and its Application in Interpolation Problem}
	\author{Dinesh Kumar Keshari, \, Suryanarayan Nayak, \, Avijit Pal, \, and \, Bhaskar Paul}

	\maketitle
	
	\begin{abstract}
In this paper, we investigate several geometric and function-theoretic properties of the domain $\mathbf{\Theta}_n$. We obtain new characterizations of its distinguished boundary and introduce the notion of a \textit{$\mathbf{\Theta}_n$-inner function}, together with several illustrative examples. We establish connections between $\mathbf{\Theta}n$-inner functions and $\Gamma_n$-inner functions, tetra-inner functions, and $\mathbf{\Theta}_{n+1}$-inner functions. Furthermore, we derive an explicit characterization of rational $\mathbf{\Theta}_n$-inner functions. As an application, for any finite collection of distinct interpolation nodes in $\mathbb{D}$ and prescribed target points in $\mathbf{\Theta}_n$, we obtain an explicit formula for the rational $\mathbf{\Theta}_n$-inner function satisfying the given interpolation data.	\end{abstract}
	
\section{Introduction}\label{Intro}

The spectral Nevanlinna-Pick interpolation problem has attracted considerable attention because of its deep connections with complex geometry, operator theory. For $$z_1,\ldots,z_m\in\mathbb D,~~~
W_1,\ldots,W_m\in M_n(\mathbb C),$$
one asks whether there exists an analytic matrix-valued function
$$
F:\mathbb D\longrightarrow M_n(\mathbb C)
$$
satisfying
$$
F(z_i)=W_i,\quad i=1,\ldots,m,$$
together with the spectral constraint
$
r(F(z))\le1, z\in\mathbb D,$
where $r(\cdot)$ denotes the spectral radius. Unlike the classical Nevanlinna-Pick problem, whose solution is completely described by positivity of the Pick matrix, the spectral version remains much more subtle. One successful approach is to replace matrix-valued interpolation by interpolation on domains determined by spectral invariants of matrices. This philosophy has led to the study of domains such as the symmetrized polydisc, the tetrablock, and several related spaces.

Let $M_n(\mathbb C)$ denote the algebra of all $n\times n$ complex matrices and let
\[
\Omega_n=\{A\in M_n(\mathbb C):r(A)<1\}
\]
be the open spectral ball. A matrix belongs to $\Omega_n$ precisely when every zero of its characteristic polynomial
\[
\operatorname{Ch}_A(z)=\det(zI-A)
\]
lies in the open unit disc. The symmetrization map
\[
\mathbf s=(s_1,\ldots,s_n):\mathbb C^n\to\mathbb C^n
\]
is defined by
\[
s_i(z_1,\ldots,z_n)
=
\sum_{1\le k_1<\cdots<k_i\le n}
z_{k_1}\cdots z_{k_i},
\qquad
1\le i\le n.
\]
Its image of $\mathbb D^n$ is the symmetrized polydisc $
G_n=\mathbf s(\mathbb D^n),$
while
$
\Gamma_n=\mathbf s(\overline{\mathbb D}^{\,n})$
is its closure.  We recall the family of polynomial maps introduced in \cite{Ghosh}, which generalize the symmetrization map. Let
$
\theta_0=1,$
and define
\[
\theta_i(z_1,\ldots,z_n)
=
s_i(z_1^m,\ldots,z_n^m),
\qquad
1\le i\le n-1,
\]
together with
\[
\theta_n(z_1,\ldots,z_n)
=
(z_1\cdots z_n)^q,
\qquad
q=\frac mp,
\]
where $p$ is a fixed divisor of $m$. The images of $\mathbb D^n$ and $\overline{\mathbb D}^{\,n}$ under $
\theta=(\theta_1,\ldots,\theta_n)$
are denoted by $\mathbf{\Theta}_n$ and $\overline{\mathbf{\Theta}}_n$, respectively.
A point $
(\theta_1,\ldots,\theta_n)\in\overline{\mathbf{\Theta}}_n$
is characterized by the property that every zero of the polynomial
\begin{equation}\label{P}
P(z)
=
z^n-\theta_1z^{n-1}
+\cdots
+(-1)^n\theta_n^p
\end{equation}
lies in $\overline{\mathbb D}$. The map $\theta$ naturally induces a holomorphic map
\[
\Pi_n:\Omega_n\longrightarrow\mathbf{\Theta}_n,
\qquad
\Pi_n(A)=\theta(\sigma(A)),
\]
which associates to each matrix the image of its spectrum under $\theta$. Conversely, every point of $\mathbf{\Theta}_n$ determines the companion matrix
$
J_n(\theta_1,\ldots,\theta_n),$
whose characteristic polynomial is exactly $P$. Thus, $\mathbf{\Theta}_n$ provides a natural parameter space for spectral data of matrices in $\Omega_n$, much as the symmetrized polydisc parametrizes unordered eigenvalues.
Another important domain closely related to the spectral ball is the tetrablock
\[
\mathbb E
=
\left\{
(x_1,x_2,x_3)\in\mathbb C^3:
1-zx_1-wx_2+zwx_3\neq0
\text{ whenever }
|z|,|w|<1
\right\},
\]
which has played a central role in $\mu$-synthesis, interpolation theory, and operator theory.

A recurring theme in the study of these domains is the role of their distinguished boundaries. For bounded polynomially convex domains, distinguished boundaries govern extremal problems, rational inner functions, and interpolation. The explicit descriptions of distinguished boundaries of the symmetrized polydisc, $\mathbf{\Theta}_n$, and the tetrablock can be found in \cite{Abouhajar,  Biswas 2, costara}.

Inner functions associated with these domains provide higher-dimensional analogues of finite Blaschke products. They arise naturally in interpolation theory and have been studied extensively for the symmetrized bidisc, the symmetrized polydisc, and the tetrablock (see \cite{JAZLNY,JAZLNY 1,Alsalhi,costara}). Motivated by these developments, we introduce the notion of $\mathbf{\Theta}_n$-inner functions and investigate their basic properties.

One motivation for this study comes from interpolation theory. In the scalar case, every solvable finite interpolation problem on the unit disc admits a Blaschke product solution. An analogous phenomenon holds for matrix-valued interpolation, where rational inner matrix functions play the role of Blaschke products. Mc-Carthy and Young \cite{McCarthy 1} proved that whenever a contractive analytic matrix-valued interpolant exists, one can choose it to be rational inner. Costara \cite{costara} further related the spectral Nevanlinna-Pick problem to the classical matrix-valued interpolation problem via similarity transformations.

We conclude the introduction with an outline of the paper. In Section~\ref{Relations}, we establish several relationships among the domains $\mathbf{\Theta}_n$, $G_n$, and the tetrablock $\mathbb{E}$, together with their distinguished boundaries. In Section~\ref{Characterizations}, we obtain a characterization of $\mathbf{\Theta}_n$ in terms of rational functions. Section~\ref{Inner Function} is devoted to the introduction and characterization of $\mathbf{\Theta}_n$-inner functions. We also construct explicit examples of $\mathbf{\Theta}_n$-inner functions and present an example of a rational function from $\mathbb{D}$ to $\mathbb{C}^2$ that is not a $\mathbf{\Theta}_2$-inner function. Furthermore, we obtain a complete description of rational $\mathbf{\Theta}_n$-inner functions. Finally, we investigate interpolation by rational $\mathbf{\Theta}_n$-inner functions and explicitly construct rational $\mathbf{\Theta}_n$-inner interpolants that solve the corresponding interpolation problem on $\mathbf{\Theta}_n$.

\section{Relations among $\mathbf{\Theta}_n$, $G_n$, and the Tetrablock}\label{Relations}

The domain $\mathbf{\Theta}_n$ is closely related to two well-studied domains in several complex variables, namely the symmetrized polydisc $G_n$ and the tetrablock $\mathbb E$. Understanding these relationships allows us to transfer several geometric and function-theoretic properties between these domains. In particular, the results established in this section will be essential in relating $\mathbf{\Theta}_n$-inner functions to $\Gamma_n$-inner functions and tetrablock-inner functions. We begin by showing that every point of $\mathbf{\Theta}_n$ naturally determines a point of the symmetrized $n$-disc.

\begin{lem}\label{Lem 1}
Let $(\theta_1,\ldots,\theta_n)\in\mathbf{\Theta}_n$. Then $
(\theta_1,\ldots,\theta_{n-1},\theta_n^p)\in G_n.$
\end{lem}

\begin{proof}
Since $(\theta_1,\ldots,\theta_n)\in\mathbf{\Theta}_n$, there exist
$z_1,\ldots,z_n\in\mathbb D$ such that
\[
\theta_i=s_i(z_1^m,\ldots,z_n^m), \qquad 1\le i\le n-1,
\]
and
\[
\theta_n=(z_1\cdots z_n)^q,
\qquad q=\frac{m}{p}.
\]
Hence
\[
\theta_n^p
=(z_1\cdots z_n)^m
=z_1^m\cdots z_n^m
=s_n(z_1^m,\ldots,z_n^m).
\]
Therefore
\[
(\theta_1,\ldots,\theta_{n-1},\theta_n^p)
=(s_1,\ldots,s_n)(z_1^m,\ldots,z_n^m)
\in G_n,
\]
which completes the proof.
\end{proof}
The proof of the next result is identical, replacing $\mathbb D$ by $\overline{\mathbb D}$.

\begin{lem}\label{Lem 1 Bar}
Let $(\theta_1,\ldots,\theta_n)\in\overline{\mathbf{\Theta}}_n$. Then
$
(\theta_1,\ldots,\theta_{n-1},\theta_n^p)\in\Gamma_n.$
\end{lem}
The preceding lemmas naturally give rise to the following mapping.
\begin{rem}\label{Rem 1}
Suppose that $p$ divides $m$. Define
\[
\pi_p:\mathbb C^n\longrightarrow\mathbb C^n,
\qquad
\pi_p(\theta_1,\ldots,\theta_n)
=(\theta_1,\ldots,\theta_{n-1},\theta_n^p).
\]
By Lemmas~\ref{Lem 1} and~\ref{Lem 1 Bar},
$
\pi_p(\mathbf{\Theta}_n)\subseteq G_n,
\pi_p(\overline{\mathbf{\Theta}}_n)\subseteq\Gamma_n.
$
Moreover, every point of $G_n$ (respectively, $\Gamma_n$) has a preimage under $\pi_p$. Consequently,
\[
\pi_p:\mathbf{\Theta}_n\to G_n
\quad\text{and}\quad
\pi_p:\overline{\mathbf{\Theta}}_n\to\Gamma_n
\]
are surjective maps.
\end{rem}

The next lemma shows how points in $\mathbf{\Theta}_n$ can be embedded naturally into $\mathbf{\Theta}_{n+1}$.

\begin{lem}\label{Lem 2}
Let $(\theta_1,\ldots,\theta_n)\in\mathbf{\Theta}_n$. Then, for every
$\alpha\in\mathbb D$,
\[
(\alpha^m+\theta_1,\,
\alpha^m\theta_1+\theta_2,\,
\ldots,\,
\alpha^m\theta_{n-1}+\theta_n^p,\,
\alpha^{m/p}\theta_n)
\in\mathbf{\Theta}_{n+1}.
\]
\end{lem}

\begin{proof}
Choose $z_1,\ldots,z_n\in\mathbb D$ such that
\[
\theta_i=s_i(z_1^m,\ldots,z_n^m),
\qquad
1\le i\le n-1,
\]
and
\[
\theta_n=(z_1\cdots z_n)^q,
\qquad
q=\frac{m}{p}.
\]
Set $z_{n+1}=\alpha$, where $\alpha\in\mathbb D$. Then

\[
\alpha^m+\theta_1
=s_1(z_1^m,\ldots,z_n^m,\alpha^m),
\]
and, for $1\le i\le n-2$,
\[
\alpha^m\theta_i+\theta_{i+1}
=s_{i+1}(z_1^m,\ldots,z_n^m,\alpha^m).
\]
Similarly,
\[
\alpha^m\theta_{n-1}+\theta_n^p
=s_n(z_1^m,\ldots,z_n^m,\alpha^m),
\]
while
\[
\alpha^{m/p}\theta_n
=(z_1\cdots z_n\alpha)^{m/p}.
\]
These are precisely the defining coordinates of a point in
$\mathbf{\Theta}_{n+1}$. Hence
\[
(\alpha^m+\theta_1,\,
\alpha^m\theta_1+\theta_2,\,
\ldots,\,
\alpha^m\theta_{n-1}+\theta_n^p,\,
\alpha^{m/p}\theta_n)
\in\mathbf{\Theta}_{n+1},
\]
as required.
\end{proof}	
We state only the corresponding result for $\overline{\mathbf{\Theta}}_n$, since its proof is identical to that of Lemma~\ref{Lem 2}.

\begin{lem}\label{Lem 2 Bar}
Let $(\theta_1,\dots,\theta_n)\in\mathbb{C}^n$. If
$(\theta_1,\dots,\theta_n)\in\overline{\mathbf{\Theta}}_n$, then
\[
(\alpha^m+\theta_1,\,
\alpha^m\theta_1+\theta_2,\,
\dots,\,
\alpha^m\theta_{n-1}+\theta_n^p,\,
\alpha^{m/p}\theta_n)
\in
\overline{\mathbf{\Theta}}_{n+1}
\]
for every $\alpha\in\overline{\mathbb D}$.
\end{lem}	
For $(\theta_1,\ldots,\theta_n)\in\mathbb C^n$ and $\alpha\in\mathbb C$, we define the functions $\Phi^{(i)}_1$ and $\Phi^{(i)}_2$, for $1\leq i\leq n-1$, by
\begin{equation}\label{Phi_1}
\begin{aligned}
\Phi^{(i)}_1(\alpha^i\theta_i,\alpha^{n-i}\theta_{n-i},\alpha^n\theta_n^p)
&=k(i)^2\bigl(1-|\alpha^n\theta_n^p|^2\bigr)
+\bigl(|\alpha^i\theta_i|^2-|\alpha^{n-i}\theta_{n-i}|^2\bigr)\\
&\quad-k(i)\alpha^i\bigl(\theta_i-|\alpha|^{2(n-i)}
\overline{\theta}_{n-i}\theta_n^p\bigr)-k(i)\overline{\alpha}^{\,i}\bigl(\overline{\theta}_i
-|\alpha|^{2(n-i)}\theta_{n-i}\overline{\theta}_n^{\,p}\bigr),
\end{aligned}
\end{equation}
and
\begin{equation}\label{Phi_2}
\begin{aligned}
\Phi^{(i)}_2(\alpha^i\theta_i,\alpha^{n-i}\theta_{n-i},\alpha^n\theta_n^p)
&=k(i)^2\bigl(1-|\alpha^n\theta_n^p|^2\bigr)
+\bigl(|\alpha^{n-i}\theta_{n-i}|^2-|\alpha^i\theta_i|^2\bigr)\\
&\quad-k(i)\alpha^{n-i}\bigl(\theta_{n-i}
-|\alpha|^{2i}\overline{\theta}_i\theta_n^p\bigr)-k(i)\overline{\alpha}^{\,n-i}\bigl(\overline{\theta}_{n-i}
-|\alpha|^{2i}\theta_i\overline{\theta}_n^{\,p}\bigr),
\end{aligned}
\end{equation}
where
$
k(i)=\binom{n-1}{i}+\binom{n-1}{n-i},~ 1\leq i\leq n-1.$
\begin{prop}\label{Prop 2}
Let $(\theta_1,\dots,\theta_n)\in\overline{\mathbf{\Theta}}_n$. Then, for every
$1\leq i\leq n-1$ and every $\alpha\in\overline{\mathbb D}$,
\begin{equation}\label{Phi^i_1, Phi^i_2 > 0}
\begin{aligned}
\Phi^{(i)}_1(\alpha^i\theta_i,\alpha^{n-i}\theta_{n-i},\alpha^n\theta_n^p)\geq0,
\qquad
\Phi^{(i)}_2(\alpha^i\theta_i,\alpha^{n-i}\theta_{n-i},\alpha^n\theta_n^p)\geq0.
\end{aligned}
\end{equation}
\end{prop}

\begin{proof}
Since $(\theta_1,\dots,\theta_n)\in\overline{\mathbf{\Theta}}_n$, Lemma~\ref{Lem 1 Bar} implies that
$
(\theta_1,\dots,\theta_{n-1},\theta_n^p)\in\Gamma_n.$
Therefore, by \cite[Theorem~3.3]{A. Pal},
\[
\Phi^{(i)}_1(\alpha^i\theta_i,\alpha^{n-i}\theta_{n-i},\alpha^n\theta_n^p)\geq0,
\qquad
\Phi^{(i)}_2(\alpha^i\theta_i,\alpha^{n-i}\theta_{n-i},\alpha^n\theta_n^p)\geq0,
\]
for every $\alpha\in\overline{\mathbb D}$ and every $1\leq i\leq n-1$.
\end{proof}

Proposition~\ref{Prop 2} immediately yields the estimate
\[
|\theta_i-\overline{\theta}_{n-i}\theta_n^p|
+
|\theta_{n-i}-\overline{\theta}_i\theta_n^p|
\leq
k(i)\bigl(1-|\theta_n^p|^2\bigr),
\]
for $1\leq i\leq n-1$. We record this as a corollary.

\begin{cor}\label{Cor 3}
Let $
k(i)=\binom{n-1}{i}+\binom{n-1}{n-i}.$
If $(\theta_1,\dots,\theta_n)\in\overline{\mathbf{\Theta}}_n$, then
\[
|\theta_i-\overline{\theta}_{n-i}\theta_n^p|
+
|\theta_{n-i}-\overline{\theta}_i\theta_n^p|
\leq
k(i)\bigl(1-|\theta_n^p|^2\bigr),
\]
for every $1\leq i\leq n-1$.
\end{cor}
Since
$
\theta_i=s_i(z_1^m,\dots,z_n^m), 1\leq i\leq n-1,$
for some $z_1,\dots,z_n\in\mathbb D$, it follows from the triangle inequality that
\begin{equation}\label{Estimate theta_i}
|\theta_i|
\leq
k(i),
\qquad
1\leq i\leq n-1.
\end{equation}

The next lemma establishes a connection between 
$\overline{\mathbf{\Theta}}_n$ and the closed tetrablock.

\begin{lem}\label{Lem 3}
Let
$
k(i)=\binom{n-1}{i}+\binom{n-1}{n-i}.$
If $(\theta_1,\dots,\theta_n)\in\overline{\mathbf{\Theta}}_n$, then
\[
\left(
\frac{\theta_i}{k(i)},
\frac{\theta_{n-i}}{k(i)},
\theta_n^p
\right)
\in
\overline{\mathbb E},
\qquad
1\leq i\leq n-1.
\]
\end{lem}

\begin{proof}
Let $(\theta_1,\dots,\theta_n)\in\overline{\mathbf{\Theta}}_n$. By Corollary~\ref{Cor 3},
\[
|\theta_i-\overline{\theta}_{n-i}\theta_n^p|
+
|\theta_{n-i}-\overline{\theta}_i\theta_n^p|
\leq
k(i)\bigl(1-|\theta_n^p|^2\bigr).
\]
Dividing both sides by $k(i)$ gives
\[
\left|
\frac{\theta_i}{k(i)}
-
\frac{\overline{\theta}_{n-i}}{k(i)}\theta_n^p
\right|
+
\left|
\frac{\theta_{n-i}}{k(i)}
-
\frac{\overline{\theta}_i}{k(i)}\theta_n^p
\right|
\leq
1-|\theta_n^p|^2.
\]
Hence, by \cite[Theorem~2.4]{Abouhajar},
$
\left(
\frac{\theta_i}{k(i)},
\frac{\theta_{n-i}}{k(i)},
\theta_n^p
\right)
\in
\overline{\mathbb E},
$
for every $1\leq i\leq n-1$.
\end{proof}
The remaining part of this section, we establish several connections between the distinguished boundaries of $\mathbf{\Theta}_n$, the symmetrized polydisc, and the tetrablock. These relationships will play an important role in the study of $\mathbf{\Theta}_n$-inner functions in the subsequent sections.

\begin{prop}\label{Prop 3}
Let $(\theta_1,\dots,\theta_n)\in\mathbb C^n$. Then the following are equivalent.
\begin{enumerate}
\item $(\theta_1,\dots,\theta_n)\in b\mathbf{\Theta}_n$.

\item $(\theta_1,\dots,\theta_{n-1},\theta_n^p)\in b\Gamma_n$.
\end{enumerate}
\end{prop}

\begin{proof}
$(1)\Rightarrow(2)$.
Suppose that $(\theta_1,\dots,\theta_n)\in b\mathbf{\Theta}_n$. By \cite[Theorem~2.5]{Biswas 2},
\[
(\theta_1,\dots,\theta_n)\in\overline{\mathbf{\Theta}}_n
\quad\text{and}\quad
|\theta_n|=1.
\]
Lemma~\ref{Lem 1 Bar} therefore implies that
$
(\theta_1,\dots,\theta_{n-1},\theta_n^p)\in\Gamma_n.$
Since $|\theta_n|=1$, we also have $|\theta_n^p|=1$. Hence, by \cite[Theorem~2.4]{Biswas},
$
(\theta_1,\dots,\theta_{n-1},\theta_n^p)\in b\Gamma_n.$

$(2)\Rightarrow(1)$:
Assume that $
(\theta_1,\dots,\theta_{n-1},\theta_n^p)\in b\Gamma_n.$
By \cite[Theorem~2.4]{Biswas},
\[
|\theta_n^p|=1,\quad
\theta_i=\overline{\theta}_{n-i}\theta_n^p,\quad
(\gamma_1\theta_1,\dots,\gamma_{n-1}\theta_{n-1})\in\Gamma_{n-1},
\]
where $\gamma_i=\frac{n-i}{n}$ for $1\leq i\leq n-1$. Since
$|\theta_n^p|=1$, it follows that $|\theta_n|=1$. Therefore,
\cite[Theorem~2.5]{Biswas 2} yields
$
(\theta_1,\dots,\theta_n)\in b\mathbf{\Theta}_n.$
This completes the proof.
\end{proof}

The next proposition shows that the distinguished boundary is preserved under the embedding constructed in Lemma~\ref{Lem 2 Bar}.

\begin{prop}\label{Prop 4}
Let $(\theta_1,\dots,\theta_n)\in b\mathbf{\Theta}_n$. Then
\[
(\alpha^m+\theta_1,\,
\alpha^m\theta_1+\theta_2,\,
\dots,\,
\alpha^m\theta_{n-1}+\theta_n^p,\,
\alpha^{m/p}\theta_n)
\in b\mathbf{\Theta}_{n+1}
\]
for every $\alpha\in\mathbb T$.
\end{prop}

\begin{proof}
Since $(\theta_1,\dots,\theta_n)\in b\mathbf{\Theta}_n$, \cite[Theorem~2.5]{Biswas 2} implies that
\[
(\theta_1,\dots,\theta_n)\in\overline{\mathbf{\Theta}}_n
\quad\text{and}\quad
|\theta_n|=1.
\]
Hence, by Lemma~\ref{Lem 2 Bar},
$
(\alpha^m+\theta_1,\,
\alpha^m\theta_1+\theta_2,\,
\dots,\,
\alpha^m\theta_{n-1}+\theta_n^p,\,
\alpha^{m/p}\theta_n)
\in\overline{\mathbf{\Theta}}_{n+1}$
for every $\alpha\in\overline{\mathbb D}$.
If $\alpha\in\mathbb T$, then
$
|\alpha^{m/p}\theta_n|=1.$
Therefore, another application of \cite[Theorem~2.5]{Biswas 2} shows that
\[
(\alpha^m+\theta_1,\,
\alpha^m\theta_1+\theta_2,\,
\dots,\,
\alpha^m\theta_{n-1}+\theta_n^p,\,
\alpha^{m/p}\theta_n)
\in b\mathbf{\Theta}_{n+1},
\]
for every $\alpha\in\mathbb T$.
\end{proof}
The next proposition relates the distinguished boundary of $\mathbf{\Theta}_n$ to that of the tetrablock.

\begin{prop}\label{Prop 5}
Let $
k(i)=\binom{n-1}{i}+\binom{n-1}{n-i}.$
Then $(\theta_1,\dots,\theta_n)\in b\mathbf{\Theta}_n$ if and only if
$
\left(
\frac{\theta_i}{k(i)},
\frac{\theta_{n-i}}{k(i)},
\theta_n^p
\right)\in b\mathbb E
$
and
$
(\gamma_1\theta_1,\dots,\gamma_{n-1}\theta_{n-1})
\in\Gamma_{n-1},$
where $\gamma_i=\frac{n-i}{n}$ for $1\leq i\leq n-1$.
\end{prop}

\begin{proof}
Suppose that $(\theta_1,\dots,\theta_n)\in b\mathbf{\Theta}_n$. By \cite[Theorem~2.5]{Biswas 2},
\[
(\theta_1,\dots,\theta_n)\in\overline{\mathbf{\Theta}}_n
\quad\text{and}\quad
|\theta_n|=1.
\]
Lemma~\ref{Lem 3}, together with \cite[Lemma~2.6]{Biswas 2}, implies that
$
\left(
\frac{\theta_i}{k(i)},
\frac{\theta_{n-i}}{k(i)},
\theta_n^p
\right)\in b\mathbb E,$
and
$
(\gamma_1\theta_1,\dots,\gamma_{n-1}\theta_{n-1})
\in\Gamma_{n-1}.$

Conversely, assume that
$
\left(
\frac{\theta_i}{k(i)},
\frac{\theta_{n-i}}{k(i)},
\theta_n^p
\right)\in b\mathbb E$
and
$
(\gamma_1\theta_1,\dots,\gamma_{n-1}\theta_{n-1})
\in\Gamma_{n-1}.$
By \cite[Theorem~7.1]{Abouhajar},
\[
\theta_i=\overline{\theta}_{n-i}\theta_n^p,
\quad
|\theta_n|=1,
\]
for $1\leq i\leq n-1$. Combining these identities with the assumption
$
(\gamma_1\theta_1,\dots,\gamma_{n-1}\theta_{n-1})
\in\Gamma_{n-1},$
we obtain, by \cite[Theorem~2.5]{Biswas 2},
$
(\theta_1,\dots,\theta_n)\in b\mathbf{\Theta}_n.$
This completes the proof.
\end{proof}

\section{Characterization of $\mathbf{\Theta}_n$ by a Rational Function on $\mathbb{D}$}\label{Characterizations}

In this section, we obtain a characterization of the domains
$\mathbf{\Theta}_n$ and $\overline{\mathbf{\Theta}}_n$
using a rational function on the unit disc.
The characterization is analogous to the classical criteria for the
symmetrized polydisc and will play an important role in the study of
$\mathbf{\Theta}_n$-inner functions developed later. Let $(\theta_1,\ldots,\theta_n)\in\mathbb C^n$, and let
$P$ be the polynomial defined in \eqref{P}. Define	
	\begin{equation}\label{Q}
		\begin{aligned}
			Q(z) &= \frac{d}{dz}(z^nP(1/z)) = n(-1)^n\theta^p_nz^{n-1} + (n-1)(-1)^{n-1}\theta_{n-1}z^{n-2} + \dots + (-\theta_1)
		\end{aligned}
	\end{equation}
	and
	\begin{equation}\label{R}
		\begin{aligned}
			R(z) &= z^{n-1}P^{\prime}(1/z) = n - (n-1)\theta_1z + \dots + (-1)^{n-1}\theta_{n-1}z^{n-1}
		\end{aligned}
	\end{equation}
	for $z \in \mathbb{C} \setminus \{0\}$. Let us define the rational function $f$ by
	\begin{equation}\label{f}
		\begin{aligned}
			f(z) &= \frac{Q(z)}{R(z)} =
			\frac{n(-1)^n\theta^p_n z^{n-1} + (n-1)(-1)^{n-1}\theta_{n-1} z^{n-2} + \dots + (- \theta_1)}{n - (n-1)\theta_1z + \dots + (-1)^{n-1}\theta_{n-1}z^{n-1}}.
		\end{aligned}
	\end{equation}
	This function plays a crucial role in characterizing $\mathbf{\Theta}_n$. The following theorem shows that membership in $\mathbf{\Theta}_n$
is completely determined by the contractivity of the rational function
$f$ on the closed unit disc.
	
	\begin{thm}\label{Theta_n Characterization}
		Let $(\theta_1, \dots, \theta_n) \in \mathbb{C}^n$ and $f$ be defined as in \eqref{f}. Then the following are equivalent:
		\begin{enumerate}
			\item $(\theta_1, \dots, \theta_n) \in \mathbf{\Theta}_n$.
			\item $(\theta_1,\ldots,\theta_{n-1},\theta_n^p)\in G_n.$
			\item $\displaystyle \sup_{|z| \leq 1} |f(z)| < 1$.
			\item $|\theta_n| < 1$ and there exists $(\alpha_1, \dots, \alpha_{n-1}) \in G_{n-1}$ such that $\theta_i = \alpha_i + \overline{
		\alpha}_{n-i}\theta^p_n$ for $1 \le i \le n-1$.
		\end{enumerate}
	\end{thm}
	
	\begin{proof}
		$(1) \Rightarrow (3):$ Suppose that $(\theta_1,\ldots,\theta_n)\in\mathbf{\Theta}_n$. Then every zero of $P$ lies in $\mathbb D$.
By Lucas theorem, every zero of $P'$ also lies in $\mathbb D$.
Consequently, $P'(z)\neq0$ for every $z\in\mathbb C\setminus\mathbb D$, and hence
$P'(1/z)\neq0$ in a neighbourhood of $\overline{\mathbb D}$.
Therefore, $R$ has no zeros in a neighbourhood of
$\overline{\mathbb D}$, so $f$ is holomorphic there. By the maximum modulus principle, it is enough to verify that
$|f(z)|<1$ for every $z\in\mathbb T$.
Since
\[
Q(z)=nz^{\,n-1}P(1/z)-z^{\,n-2}P'(1/z),
\]
we obtain
		\begin{equation}\label{E 2.1}
			\begin{aligned}
				|f(z)| &=
				\left| \frac{nz^{n-1}P(1/z) - z^{n-2}P^{\prime}(1/z)}{z^{n-1}P^{\prime}(1/z)} \right|\\
				&= \left| n\frac{P(1/z)}{P^{\prime}(1/z)} - \frac{1}{z} \right|\\
				&= \left| n\frac{P(\bar{z})}{\bar{z}P^{\prime}(\bar{z})} - 1 \right|
			\end{aligned}
		\end{equation}
		for every $z\in \mathbb T.$ We show that $z \mapsto n\frac{P(z)}{zP^{\prime}(z)}$ takes $\mathbb{T}$ into $\{z \in \mathbb{C} : |z - 1| < 1\}$. Suppose that $w_1, \dots, w_n$ are the zeros of $P$. Then, for all $z \in \mathbb{T},$ we have,
		\begin{equation}\label{E 2.2}
			\begin{aligned}
				\frac{zP^{\prime}(z)}{nP(z)} &= \frac{1}{n} \sum_{i=1}^{n} \frac{z}{z - w_i} = \frac{1}{n} \sum_{i=1}^n \frac{1}{1 - w_i/z}.
			\end{aligned}
		\end{equation}
Hence,
$\frac{zP'(z)}{nP(z)}$
lies in the convex hull of the points
$
\left\{\frac{1}{1-w_i/z}:1\le i\le n\right\}.$
Since $|w_i/z|<1$ for every $z\in\mathbb T$,
the conformal map
\[
w\longmapsto\frac{1}{1-w}
\]
maps $\mathbb D$ onto the half-plane
$\{w\in\mathbb C:\operatorname{Re}w>\frac12\}$.
Therefore
$
\operatorname{Re}\left(\frac{zP'(z)}{nP(z)}\right)>\frac12.$
Applying the inverse conformal map gives
$
\left|
n\frac{P(z)}{zP'(z)}-1
\right|<1,$
which proves that
$|f(z)|<1$ on $\mathbb T$.
Hence
$
\sup_{|z|\le1}|f(z)|<1.$	
		
$(3) \Rightarrow (1):$ Suppose that
$
\sup_{|z|\le1}|f(z)|<1.$
We first prove that $R$ has no zeros on
$\overline{\mathbb D}$.
Assume, to the contrary, that
$R(z_0)=0$ for some
$z_0\in\overline{\mathbb D}$.
Since $R(0)=n$, necessarily $z_0\neq0$.
Because $f$ is bounded by one on
$\overline{\mathbb D}$,
we must also have
$Q(z_0)=0$.  From $R(z_0) = 0$ we get $P^{\prime}(1/z_0) = 0$. Since $Q(z_0)=0$, we have
\[
nz_0^{\,n-1}P(1/z_0)-z_0^{\,n-2}P'(1/z_0)=0.
\]
As $P'(1/z_0)=0$, it follows that
$
nz_0^{\,n-1}P(1/z_0)=0.$
Since $z_0\neq 0$, we conclude that $P(1/z_0)=0$. Hence $1/z_0$ is a common zero of $P$ and $P'$, and therefore it is a zero of $P$ of multiplicity $m\ge2$. Accordingly, we may write
\[
P(z)=\left(z-\frac1{z_0}\right)^m h(z),
\]
where $h$ is analytic and non-vanishing in a neighbourhood of $1/z_0$. Substituting this factorization into the definition of $f$, we obtain, in a neighbourhood of $z_0$,
\begin{equation}\label{E 2.3}
\begin{aligned}
f(z)
&=\frac{nz^{n-1}P(1/z)-z^{n-2}P'(1/z)}
        {z^{n-1}P'(1/z)}\\
&=\frac{nz^{n-1}(1/z-1/z_0)^m h(1/z)
-z^{n-2}\!\left(m(1/z-1/z_0)^{m-1}h(1/z)
+(1/z-1/z_0)^m h'(1/z)\right)}
{z^{n-1}\!\left(m(1/z-1/z_0)^{m-1}h(1/z)
-(1/z-1/z_0)^m h'(1/z)\right)}\\
&=\frac{nz^{n-1}(1/z-1/z_0)h(1/z)
-z^{n-2}\!\left(mh(1/z)
+(1/z-1/z_0)h'(1/z)\right)}
{z^{n-1}\!\left(m(1/z-1/z_0)h(1/z)
-(1/z-1/z_0)h'(1/z)\right)}.
\end{aligned}
\end{equation}		Observe that
		\begin{equation*}
			\begin{aligned}
				|f(z_0)| &=
				\left|\frac{-mz^{n-2}_0h(1/z_0)}{mz^{n-1}_0h(1/z_0)}\right| = |1/z_0| \ge 1.
			\end{aligned}
		\end{equation*}
This contradiction shows that
$R(z)\neq0$ for every
$z\in\overline{\mathbb D}$.
Consequently,
$
|Q(z)|<|R(z)|,z\in\mathbb T.$
Replacing $z$ by $1/z$ and multiplying by $|z|^{\,n}$
(which equals $1$ on $\mathbb T$)
gives
\begin{equation}\label{E 2.4}
			\begin{aligned}
				&|n(-1)^n\theta^p_n + (n-1)(-1)^{n-1}\theta_{n-1}z + \dots + (-\theta_1)z^{n-1}|\\
				&\quad <
				|nz^n - (n-1)\theta_1z^{n-1} + \dots + (-1)^{n-1}\theta_{n-1}z| \,\, \text{for all} \,\, z \in \mathbb{T}.
			\end{aligned}
		\end{equation}
		It now follows from Rouch\'e's theorem that the polynomials
\begin{equation}\label{E 2.5}
\begin{aligned}
&(n(-1)^n\theta_n^p +(n-1)(-1)^{n-1}\theta_{n-1}z+\cdots-\theta_1z^{n-1})\\
&\qquad +(nz^n-(n-1)\theta_1z^{n-1}+\cdots+(-1)^{n-1}\theta_{n-1}z)
\end{aligned}
\end{equation}
and
\begin{equation}\label{E 2.6}
nz^n-(n-1)\theta_1z^{n-1}+\cdots+(-1)^{n-1}\theta_{n-1}z
\end{equation}
have the same number of zeros in $\mathbb D$. Observe that the polynomial in \eqref{E 2.6} is precisely
$
z^nR(1/z).$
Since $R$ has no zeros in $\overline{\mathbb D}$, it follows that
$z^nR(1/z)$ has exactly $n$ zeros (counting multiplicity), all of which lie in $\mathbb D$.
On the other hand, the polynomial in \eqref{E 2.5} is simply $nP$.
Hence, by Rouch\'e's theorem, $nP$ also has $n$ zeros in $\mathbb D$.
Since $P$ is a polynomial of degree $n$, all of its zeros belong to $\mathbb D$.
Therefore, $
(\theta_1,\ldots,\theta_n)\in\mathbf{\Theta}_n.$

The equivalence of (2) and (3) follows from \cite[Theorem 3.1]{costara}.
Also, the equivalence of  $(2)$ and $(4)$ follows from \cite[Theorem 3.6]{costara}. 
This completes the proof.
	\end{proof}
	
	We now prove an analogous result for $\overline{\mathbf{\Theta}}_n$.
	
	\begin{thm}\label{Bar Theta_n Charaterization}
Let $(\theta_1,\ldots,\theta_n)\in\mathbb C^n$ and let $f$ be the rational function defined in \eqref{f}. Then the following are equivalent.
\begin{enumerate}
\item $(\theta_1,\ldots,\theta_n)\in\overline{\mathbf{\Theta}}_n$.
\item $(\theta_1,\ldots,\theta_{n-1},\theta_n^p)\in \Gamma_n.$
\item $\displaystyle\sup_{|z|\le1}|f(z)|\le1$.
\item $|\theta_n| \leq1$ and there exists $(\alpha_1, \dots, \alpha_{n-1}) \in \Gamma_{n-1}$ such that $\theta_i = \alpha_i + \overline{
		\alpha}_{n-i}\theta^p_n$ for $1 \le i \le n-1$.
\end{enumerate}
\end{thm}	
	\begin{proof}
$(1) \Rightarrow (3):$ Suppose that $(\theta_1,\ldots,\theta_n)\in\overline{\mathbf{\Theta}}_n$.
Then every zero of $P$ lies in $\overline{\mathbb D}$.
Hence, by Lucas' theorem, every zero of $P'$ also lies in
$\overline{\mathbb D}$.
Consequently, $R(z)\neq0$ for every $z\in\mathbb D$. For $0<r<1$, the point
$
(r\theta_1,r^2\theta_2,\ldots,r^{n-1}\theta_{n-1},r^{n/p}\theta_n)$
belongs to $\mathbf{\Theta}_n$.
Therefore, by Theorem \ref{Theta_n Characterization},
\[
\left|
\frac{
n(-1)^nr^n\theta_n^pz^{n-1}
+(n-1)(-1)^{n-1}r^{n-1}\theta_{n-1}z^{n-2}
+\cdots-r\theta_1}
{
n-(n-1)r\theta_1z+\cdots+(-1)^{n-1}r^{n-1}\theta_{n-1}z^{n-1}}
\right|<1
\]
for every $z\in\overline{\mathbb D}$. Since $R(z)\neq0$ for $z\in\mathbb D$, letting $r\to1^{-}$ gives
$
|f(z)|\le1,\quad z\in\mathbb D.$
As $f$ is rational and has no poles in $\mathbb D$, the inequality extends to
$\overline{\mathbb D}$ by continuity. Hence
$
\sup_{|z|\le1}|f(z)|\le1.$	
		
$(3) \Rightarrow (1):$ Suppose that $
\sup_{|z|\le1}|f(z)|\le1.$
Let $0<r<1$. Then
\begin{align*}
&\sup_{|z|\le1}
\left|
\frac{
n(-1)^nr^n\theta_n^pz^{n-1}
+(n-1)(-1)^{n-1}r^{n-1}\theta_{n-1}z^{n-2}
+\cdots-r\theta_1}
{
n-(n-1)r\theta_1z+\cdots+(-1)^{n-1}r^{n-1}\theta_{n-1}z^{n-1}}
\right|  \\
&\qquad
=r\sup_{|w|\le r}|f(w)|
\le r<1.
\end{align*}
Hence, by Theorem \ref{Theta_n Characterization},
$
(r\theta_1,r^2\theta_2,\ldots,r^{n-1}\theta_{n-1},r^{n/p}\theta_n)
\in\mathbf{\Theta}_n.$
Letting $r\to1^{-}$ and using the closedness of
$\overline{\mathbf{\Theta}}_n$, we conclude that
$
(\theta_1,\ldots,\theta_n)\in\overline{\mathbf{\Theta}}_n.$

The equivalence of $(2)$ and $(3)$ follows from \cite[Theorem 3.2]{costara}. 
Also, the equivalence of  $(2)$ and $(4)$ follows from \cite[Theorem 3.7]{costara}. 
This completes the proof.
	
	\end{proof}
	
If $(\theta_1,\ldots,\theta_n)\in
\overline{\mathbf{\Theta}}_n\setminus\mathbf{\Theta}_n$,
then $R$ has no zeros in $\mathbb D$, although it may have zeros on
$\mathbb T$.
In that case, every zero of $R$ on $\mathbb T$ is also a zero of $Q$,
with at least the same multiplicity.
Hence the common factors of $Q$ and $R$ corresponding to zeros on
$\mathbb T$ cancel, and the resulting reduced rational function is
analytic on a neighbourhood of $\overline{\mathbb D}$.
Combining Theorems \ref{Theta_n Characterization}
and \ref{Bar Theta_n Charaterization}, we obtain a characterization
of the topological boundary
$\partial\mathbf{\Theta}_n$. The next corollary follows immediately from
Theorems~\ref{Theta_n Characterization} and
\ref{Bar Theta_n Charaterization}.
	\begin{cor}\label{Characterization of Bd Theta_n}
		Let $(\theta_1, \dots, \theta_n) \in \mathbb{C}^n$ and $f$ be as in \eqref{f}. Then the following are equivalent:
		\begin{enumerate}
			\item $(\theta_1, \dots, \theta_n) \in \partial\mathbf{\Theta}_n$.
			
			\item $\displaystyle \sup_{|z| \leq 1} |f(z)| = 1$.
		\end{enumerate}
	\end{cor}
	\begin{proof}
This follows immediately from Theorems
\ref{Theta_n Characterization} and
\ref{Bar Theta_n Charaterization}.
Indeed, $
\partial\mathbf{\Theta}_n
=
\overline{\mathbf{\Theta}}_n\setminus\mathbf{\Theta}_n,$
so $(\theta_1,\ldots,\theta_n)\in\partial\mathbf{\Theta}_n$ if and only if
\[\displaystyle
\sup_{|z|\le1}|f(z)|\le1
\quad\text{and}\quad
\sup_{|z|\le1}|f(z)|\not<1,
\]
which is equivalent to $
\displaystyle\sup_{|z|\le1}|f(z)|=1.$
\end{proof}
The following corollary provides a new characterization of the distinguished
boundary of $\mathbf{\Theta}_n$.

\begin{cor}\label{Characterization of bTheta_n}
Let $(\theta_1,\dots,\theta_n)\in\mathbb C^n$ and let $f$ be the rational
function defined by \eqref{f}. Then the following are equivalent.
\begin{enumerate}
\item $(\theta_1,\dots,\theta_n)\in b\mathbf{\Theta}_n$.

\item $|\theta_n|=1$ and $
|f(z)|=1, z\in\mathbb T.$
\end{enumerate}
\end{cor}

\begin{proof}
Suppose that $(\theta_1,\dots,\theta_n)\in b\mathbf{\Theta}_n$.
By \cite[Theorem~2.5]{Biswas 2},
\[
|\theta_n|=1,\quad
\theta_i=\overline{\theta}_{\,n-i}\theta_n^p,\quad
1\le i\le n-1.
\]
Moreover, $
(\gamma_1\theta_1,\dots,\gamma_{n-1}\theta_{n-1})\in\Gamma_{n-1},$
where $\gamma_i=\frac{n-i}{n}$.
Since the denominator of $f$ has no zeros on $\mathbb T$, for every
$z\in\mathbb T$ we obtain
\begin{align}
|f(z)|
&=
\left|
\frac{n(-1)^n\theta_n^pz^{\,n-1}
+(n-1)(-1)^{n-1}\theta_{n-1}z^{\,n-2}
+\cdots-\theta_1}
{n-(n-1)\theta_1z+\cdots+(-1)^{n-1}\theta_{n-1}z^{\,n-1}}
\right| \nonumber\\
&=
\left|
\frac{n(-1)^n\theta_n^pz^{\,n-1}
+(n-1)(-1)^{n-1}\overline{\theta}_1\theta_n^pz^{\,n-2}
+\cdots-\overline{\theta}_{\,n-1}\theta_n^p}
{n-(n-1)\theta_1z+\cdots+(-1)^{n-1}\theta_{n-1}z^{\,n-1}}
\right| \nonumber\\
&=
\left|
\frac{(-1)^n\theta_n^pz^{\,n-1}
\left(
n-(n-1)\overline{\theta}_1\overline z+\cdots
+(-1)^{n-1}\overline{\theta}_{\,n-1}\overline z^{\,n-1}
\right)}
{n-(n-1)\theta_1z+\cdots+(-1)^{n-1}\theta_{n-1}z^{\,n-1}}
\right| \nonumber\\
&=1,
\end{align}
because $|z|=|\theta_n|=1$.

Conversely, suppose that $|\theta_n|=1$ and
$|f(z)|=1$ for every $z\in\mathbb T$.
Since $f$ is analytic on $\mathbb D$, the maximum modulus principle
implies that
$
|f(z)|\le1,z\in\mathbb D.$
Hence,
$
\sup_{|z|\le1}|f(z)|\le1.$
By Theorem~\ref{Bar Theta_n Charaterization},
$
(\theta_1,\dots,\theta_n)\in\overline{\mathbf{\Theta}}_n.$
Finally, since $|\theta_n|=1$, the characterization of
$b\mathbf{\Theta}_n$ in \cite[Theorem~2.5]{Biswas 2} shows that
$
(\theta_1,\dots,\theta_n)\in b\mathbf{\Theta}_n.$
This completes the proof.
\end{proof}

	\section{$\mathbf{\Theta}_n$-Inner Functions}\label{Inner Function}

Inner functions play a fundamental role in complex function theory and operator theory, with important applications to interpolation, invariant subspaces, and model theory. Motivated by the classical notion of inner functions on the unit disc and its analogues on the symmetrized polydisc and the tetrablock, we introduce the class of $\mathbf{\Theta}_n$-inner functions. In this section, we define $\mathbf{\Theta}_n$-inner functions and investigate their basic properties. We establish connections between $\mathbf{\Theta}_n$-inner functions, $\Gamma_n$-inner functions \cite{costara}, and $\mathbb{E}$-inner (tetrablock-inner) functions \cite{Alsalhi}. We also construct explicit examples of $\mathbf{\Theta}_n$-inner functions and conclude with a complete characterization of rational $\mathbf{\Theta}_n$-inner functions.

\begin{defn}\label{Theta_n inner function}
A holomorphic map
$
\Theta=(\theta_1,\dots,\theta_n):\mathbb D\to \mathbf{\Theta}_n$
is called a \emph{$\mathbf{\Theta}_n$-inner function} if its radial boundary values satisfy
\[
\lim_{r\to1^-}\Theta(rz)\in b\mathbf{\Theta}_n
\quad\text{for almost every } z\in\mathbb T,
\]
with respect to the Lebesgue measure on $\mathbb T$.
\end{defn}
Throughout this section, we write
$
\theta_n^p(z):=(\theta_n(z))^p.$
	\subsection{Examples of $\mathbf{\Theta}_n$-Inner Functions}
We begin with a basic class of examples arising from finite Blaschke products.

\begin{exam}\label{Exam 1}
Let $B_1,\ldots,B_n$ be finite Blaschke products. Define the holomorphic map
$
\Theta_B:\mathbb D\longrightarrow \overline{\mathbf{\Theta}}_n$
by
\begin{equation}\label{Theta_B}
\Theta_B(z)
=
\big(
s_1(B_1^m(z),\ldots,B_n^m(z)),
\ldots,
s_{n-1}(B_1^m(z),\ldots,B_n^m(z)),
(B_1(z)\cdots B_n(z))^{m/p}
\big).
\end{equation}
We claim that $\Theta_B$ is a $\mathbf{\Theta}_n$-inner function. Indeed, since each finite Blaschke product maps $\mathbb D$ into itself, we have
\[
B_i(z)\in\mathbb D,\quad 1\le i\le n,\ z\in\mathbb D.
\]
Hence, $
\Theta_B(z)\in\mathbf{\Theta}_n,$
for every $z\in\mathbb D$. Therefore, $\Theta_B$ is a holomorphic map from $\mathbb D$ into $\overline{\mathbf{\Theta}}_n$.
Moreover, every finite Blaschke product is unimodular almost everywhere on the unit circle. Consequently,
\begin{equation}\label{Ex1}
|B_1(\zeta)\cdots B_n(\zeta)|
=
|B_1(\zeta)|\cdots |B_n(\zeta)|
=1,
\quad \text{for all}~ \zeta\in\mathbb T.
\end{equation}
It follows that
\[
|\theta_n(\zeta)|=1,
\quad \text{for a.e. } \zeta\in\mathbb T.
\]
Since $\Theta_B(\zeta)\in\overline{\mathbf{\Theta}}_n$ and its last coordinate has modulus one, we conclude from the characterization of the distinguished boundary that
\[
\Theta_B(\zeta)\in b\mathbf{\Theta}_n,
\quad \text{for a.e. } \zeta\in\mathbb T.
\]
Hence, $\Theta_B$ is a $\mathbf{\Theta}_n$-inner function.
In particular, when $n=2$, if $B_1$ and $B_2$ are finite Blaschke products, then
\[
\big(B_1^m+B_2^m,\,(B_1B_2)^{m/p}\big)
\]
is a $\mathbf{\Theta}_2$-inner function. Furthermore, choosing
\[
B_1(z)=B_2(z)=z,\qquad m=p=1,
\]
we obtain the classical $\Gamma$-inner function
$
(2z,z^2).$
\end{exam}
Using the parametrization of $\overline{\mathbf{\Theta}}_n$, we obtain another class of
$\mathbf{\Theta}_n$-inner functions.

\begin{exam}\label{Exam 2}
Let $
(\alpha_1,\ldots,\alpha_{n-1})\in\Gamma_{n-1}.$
Define $
\widetilde{\Theta}:\mathbb D\longrightarrow\overline{\mathbf{\Theta}}_n$
by
\begin{equation}\label{Theta tilde}
\widetilde{\Theta}(z)
=
\bigl(
\alpha_1+\overline{\alpha}_{n-1}z^p,
\alpha_2+\overline{\alpha}_{n-2}z^p,
\ldots,
\alpha_{n-2}+\overline{\alpha}_2z^p,
\alpha_{n-1}+\overline{\alpha}_1z^p,
z
\bigr).
\end{equation}
We claim that $\widetilde{\Theta}$ is a $\mathbf{\Theta}_n$-inner function.
Let
$
\widetilde{\theta}_i(z)
=
\alpha_i+\overline{\alpha}_{n-i}z^p,
1\le i\le n-1,$
and
$
\widetilde{\theta}_n(z)=z.$
Clearly, $
|\widetilde{\theta}_n(z)|\le1,
 z\in\mathbb D.$
Since
$
(\alpha_1,\ldots,\alpha_{n-1})\in\Gamma_{n-1},$
it follows from part~(4) of Theorem~\ref{Bar Theta_n Charaterization} that
$
(\widetilde{\theta}_1(z),\ldots,\widetilde{\theta}_n(z))
\in
\overline{\mathbf{\Theta}}_n,
z\in\mathbb D.$
Hence,
$
\widetilde{\Theta}(\mathbb D)
\subseteq
\overline{\mathbf{\Theta}}_n.$
Moreover,
\[
|\widetilde{\theta}_n(z)|=1,
\qquad
z\in\mathbb T.
\]
Therefore, by \cite[Theorem~2.5]{Biswas 2},
$
\widetilde{\Theta}(z)\in b\mathbf{\Theta}_n$
for almost every \(z\in\mathbb T\). Consequently,
\(\widetilde{\Theta}\) is a
\(\mathbf{\Theta}_n\)-inner function.
\end{exam}

We now present another construction of a $\mathbf{\Theta}_2$-inner function using matrix-valued inner functions. Although the resulting function is closely related to Example \ref{Exam 1}, the construction is different and is based on unitary equivalence of matrix-valued functions. To this end, define
\[
\pi:M_2(\mathbb C)\longrightarrow\mathbb C^2,
\qquad
\pi(A)=\left(\operatorname{tr}A,(\det A)^{1/p}\right).
\]

\begin{exam}\label{Exam 3}

Let $\phi$ and $\psi$ be two inner functions. Define
$
h:\mathbb D\longrightarrow\mathbb B_{2\times2}$
by
\[
h(z)=
\begin{pmatrix}
\phi^m(z)&0\\
0&\psi^m(z)
\end{pmatrix}.
\]
Since $\phi$ and $\psi$ are inner,
\[
\|h(z)\|
=\max\{|\phi^m(z)|,|\psi^m(z)|\}
<1,
\qquad z\in\mathbb D.
\]
Hence $h$ is a holomorphic map from $\mathbb D$ into the matrix ball.
Moreover,
\[
\pi(h(z))
=
\left(
\phi^m(z)+\psi^m(z),
(\phi(z)\psi(z))^{m/p}
\right),
\]
which is the $\mathbf{\Theta}_2$-inner function obtained in Example \ref{Exam 1}.
We now obtain another realization of the same function. Let
\[
U=
\frac1{\sqrt2}
\begin{pmatrix}
1&1\\
-1&1
\end{pmatrix},
\quad
V=I_2,
\]
and define
$
g_1(z)=Uh(z)V.$
Since multiplication by unitary matrices preserves the operator norm,
$g_1$ also maps $\mathbb D$ into $\mathbb B_{2\times2}$.
A direct computation gives
\[
g_1(z)
=
\frac1{\sqrt2}
\begin{pmatrix}
\phi^m(z)&\psi^m(z)\\
-\phi^m(z)&\psi^m(z)
\end{pmatrix}.
\]
Therefore,
\[
\Theta'(z)
=\pi(g_1(z))
=
\left(
\frac{\phi^m(z)+\psi^m(z)}{\sqrt2},
(\phi(z)\psi(z))^{m/p}
\right).
\]
Since
\[
\det g_1(z)
=\phi^m(z)\psi^m(z),
\]
the second coordinate is indeed $(\phi\psi)^{m/p}$.
Let
\[
\theta_1(z)=\frac{\phi^m(z)+\psi^m(z)}{\sqrt2},
\quad
\theta_2(z)=(\phi(z)\psi(z))^{m/p}.
\]
Because $\phi$ and $\psi$ are inner,
$\Theta'(z)\in\overline{\mathbf{\Theta}}_2$ for every $z\in\mathbb D$.
Furthermore, for almost every $z\in\mathbb T$,
\[
\theta_1(z)
=
\overline{\theta_1(z)}\,\theta_2(z)^p,
\quad
|\theta_2(z)|=1.
\]
Hence $\Theta'$ takes boundary values in
$b\mathbf{\Theta}_2$, and therefore
$\Theta'$ is a $\mathbf{\Theta}_2$-inner function.
\end{exam}
We now present an example of a holomorphic map
$\Theta^{\times}:\mathbb D\to\overline{\mathbf{\Theta}}_2$
which is not a $\mathbf{\Theta}_2$-inner function. This shows that
not every choice of pair  of unitary matrices $(U,V)$ in the construction of
Example \ref{Exam 3} produces a $\mathbf{\Theta}_2$-inner function.

\begin{exam}
Let $m=6$ and $p=3$. Let $\phi$, $\psi$, $h$, and $U$ be as in
Example \ref{Exam 3}, and choose
\[
V=
\begin{pmatrix}
0&1\\
1&0
\end{pmatrix}.
\]
Define
$
g_2(z)=Uh(z)V.$
A direct computation gives
\[
g_2(z)
=
\frac1{\sqrt2}
\begin{pmatrix}
\psi^6(z)&\phi^6(z)\\
\psi^6(z)&-\phi^6(z)
\end{pmatrix}.
\]
Let $
\Theta^{\times}=\pi\circ g_2.$
Then
\begin{equation}\label{Ex4}
\Theta^{\times}(z)
=
\left(
\frac{\psi^6(z)-\phi^6(z)}{\sqrt2},
(-\phi(z)\psi(z))^{6/3}
\right)
=
\left(
\frac{\psi^6(z)-\phi^6(z)}{\sqrt2},
\phi^2(z)\psi^2(z)
\right).
\end{equation}
Since $\phi$ and $\psi$ are inner functions,
$\Theta^{\times}$ maps $\mathbb D$ into
$\overline{\mathbf{\Theta}}_2$.
However, $\Theta^{\times}$ is not a
$\mathbf{\Theta}_2$-inner function.
Indeed, let
\[
\theta_1(z)
=
\frac{\psi^6(z)-\phi^6(z)}{\sqrt2},
\quad
\theta_2(z)
=
\phi^2(z)\psi^2(z).
\]
For almost every $z\in\mathbb T$, we have
\[
\begin{aligned}
\overline{\theta_1(z)}\,\theta_2(z)^3
&=
\frac1{\sqrt2}
(\overline{\psi(z)}^{\,6}-\overline{\phi(z)}^{\,6})
\phi^6(z)\psi^6(z)\\
&=
\frac1{\sqrt2}
(\phi^6(z)-\psi^6(z))\\
&=
-\theta_1(z).
\end{aligned}
\]
Hence
$
\overline{\theta_1(z)}\,\theta_2(z)^3
\neq
\theta_1(z)$
for almost every $z\in\mathbb T$. Therefore,
$\Theta^{\times}$ does not take boundary values in
$b\mathbf{\Theta}_2$, and consequently it is not a
$\mathbf{\Theta}_2$-inner function.
\end{exam}	
	\subsection{Properties of $\mathbf{\Theta}_n$-Inner Functions}

In this subsection, we establish several basic properties that every
$\mathbf{\Theta}_n$-inner function must satisfy. We begin by recalling
some standard notation for reflected polynomials.
Let $f$ be a polynomial of degree at most $n$, where $n\geq 0$. We define
the \emph{reflected polynomial} $f^{\sim n}$ by
\begin{equation}\label{f sim}
\begin{aligned}
f^{\sim n}(z)
&=
z^n\overline{f(1/\overline{z})}.
\end{aligned}
\end{equation}
We also define
\begin{equation}\label{f vee}
\begin{aligned}
f^{\vee}(z)
&=
\overline{f(\overline{z})}.
\end{aligned}
\end{equation}
It follows immediately from the definitions that
\begin{equation}\label{f sim, f vee}
\begin{aligned}
f^{\sim n}(z)
=
z^nf^{\vee}(1/z).
\end{aligned}
\end{equation}
Moreover, if $f$ has degree $k\le n$, then
\begin{equation}\label{f sim sim}
\begin{aligned}
(f^{\sim n})^{\sim n}(z)=f(z).
\end{aligned}
\end{equation}

We now derive some necessary conditions for a holomorphic map
$\Theta=(\theta_1,\ldots,\theta_n)$ to be a
$\mathbf{\Theta}_n$-inner function.

\begin{lem}\label{Lem 4}
Let $\Theta=(\theta_1,\ldots,\theta_n)$ be a
$\mathbf{\Theta}_n$-inner function. Then the following hold.

\begin{enumerate}
\item
For almost every $z\in\mathbb T$,
$
\theta_i(z)=\overline{\theta_{n-i}(z)}\,\theta_n(z)^p,
|\theta_n(z)|=1,$
for $1\le i\le n-1$.

\item
For every $z\in\mathbb D$,
$
|\theta_i(z)|
\le
k(i),$
where
$
k(i)=\binom{n-1}{i}+\binom{n-1}{n-i},
1\le i\le n-1.$

\item
The function $\theta_n$ is an inner function.
\end{enumerate}
\end{lem}	
\begin{proof}
Since $\Theta=(\theta_1,\ldots,\theta_n)$ is a $\mathbf{\Theta}_n$-inner function,
\[
\Theta(z)\in b\mathbf{\Theta}_n
\qquad\text{for a.e. }z\in\mathbb T.
\]
Therefore, by \cite[Theorem 2.5]{Biswas 2},
$
\theta_i(z)=\overline{\theta_{n-i}(z)}\,\theta_n(z)^p,
|\theta_n(z)|=1,$
for $1\le i\le n-1$ and for almost every $z\in\mathbb T$. This proves part (1).

For part (2), note that $
\Theta(\mathbb D)\subseteq\overline{\mathbf{\Theta}}_n.$
Hence, by the estimate \eqref{Estimate theta_i},
$
|\theta_i(z)|
\le
k(i),
1\le i\le n-1,
$
where
$
k(i)=\binom{n-1}{i}+\binom{n-1}{n-i}.$

Finally, part (1) shows that
\[
|\theta_n(z)|=1
\quad\text{for a.e. }z\in\mathbb T.
\]
Since $\theta_n$ is holomorphic on $\mathbb D$, it follows that $\theta_n$ is an inner function. This completes the proof.
\end{proof}

\begin{lem}\label{Lem 5}
Let $\Theta=(\theta_1,\ldots,\theta_n)$ be a rational
$\mathbf{\Theta}_n$-inner function. Then
\[
\theta_i(z)
=
\theta_{n-i}^{\vee}(1/z)\theta_n(z)^p,
\qquad
z\in\mathbb C,
\]
for $1\le i\le n-1$.
\end{lem}

\begin{proof}
By Lemma \ref{Lem 4},
\[
\theta_i(z)
=
\overline{\theta_{n-i}(z)}\,\theta_n(z)^p
=
\theta_{n-i}^{\vee}(\overline z)\,\theta_n(z)^p
=
\theta_{n-i}^{\vee}(1/z)\,\theta_n(z)^p
\]
for almost every $z\in\mathbb T$, where the last equality follows from the identity
$\overline z=1/z$ on $\mathbb T$.
Since $\Theta$ is rational, both
$
\theta_i(z)~\text{and}~
\theta_{n-i}^{\vee}(1/z)\theta_n(z)^p$
are rational functions. As they agree on a subset of $\mathbb T$ of positive measure, they agree on a set having an accumulation point. Hence, by the identity theorem, they coincide identically. Therefore,
\[
\theta_i(z)
=
\theta_{n-i}^{\vee}(1/z)\theta_n(z)^p
\]
for all $z\in\mathbb C$. This completes the proof.
\end{proof}

	\begin{prop}\label{Prop 6}
		Let $\Theta = (\theta_1, \dots, \theta_n)$ be a rational $\mathbf{\Theta}_n$-inner function.
		\begin{enumerate}
			\item If $a \in \mathbb{C} \cup \{\infty\}$ is a pole of $\theta_n$, of multiplicity $k \ge 0$ and $1/\overline{a}$ is a zero of $\theta_{n-i}$ of multiplicity $l \ge 0$, then $a$ is a pole of $\theta_i$ of multiplicity at least $pk - l$.
			
			\item If $a \in \mathbb{C} \cup \{\infty\}$ is a pole of $\theta_i$, of multiplicity $k \ge 0$ then $a$ is a pole of $\theta^p_n$ multiplicity at least $k$.
		\end{enumerate}
	\end{prop}
	\begin{proof}
We consider separately the cases $a\in\mathbb C$ and $a=\infty$.

\medskip

\noindent
\textbf{(1) Case $a\in\mathbb C$:}
By Lemma \ref{Lem 5},
\[
\theta_i(z)
=
\theta_{n-i}^{\vee}(1/z)\theta_n(z)^p.
\]
Since $\theta_n$ is an inner function, it is holomorphic on
$\overline{\mathbb D}$. Hence every pole of $\theta_n$ lies outside
$\overline{\mathbb D}$, so $|a|>1$. Consequently,
$|1/a|<1$, and therefore $\theta_{n-i}^{\vee}$ is holomorphic at $1/a$.
If $1/\overline a$ is a zero of $\theta_{n-i}$ of multiplicity $l$, then
$\theta_{n-i}^{\vee}(1/z)$ has a zero at $a$ of multiplicity $l$.
Since $a$ is a pole of $\theta_n$ of multiplicity $k$,
$\theta_n^p$ has a pole at $a$ of multiplicity $pk$.
Hence,
\[
\theta_i
=
\theta_{n-i}^{\vee}(1/z)\theta_n^p
\]
has a pole at $a$ of multiplicity at least $pk-l$.

\medskip

\noindent
\textbf{Case $a=\infty$:}
If $\theta_n$ has a pole at $\infty$ of multiplicity $k$, then
$\theta_n(1/z)^p$ has a pole at $0$ of multiplicity $pk$.
Similarly, if $0$ is a zero of $\theta_{n-i}$ of multiplicity $l$, then
$\theta_{n-i}^{\vee}(z)$ has a zero of multiplicity $l$ at $0$.
Using
\[
\theta_i(1/z)
=
\theta_{n-i}^{\vee}(z)\theta_n(1/z)^p,
\]
we conclude that $\theta_i(1/z)$ has a pole at $0$ of multiplicity at least
$pk-l$. Equivalently, $\theta_i$ has a pole at $\infty$ of multiplicity at least
$pk-l$.

\medskip

\noindent
\textbf{(2) Case $a\in\mathbb C$:}
Suppose $a$ is a pole of $\theta_i$ of multiplicity $k$.
Since $|a|>1$, the function $\theta_{n-i}^{\vee}(1/z)$ is holomorphic at $a$.
Again using
\[
\theta_i(z)
=
\theta_{n-i}^{\vee}(1/z)\theta_n(z)^p,
\]
it follows immediately that $\theta_n^p$ has a pole at $a$ of multiplicity at least
$k$.

\medskip

\noindent
\textbf{Case $a=\infty$:}
If $\theta_i$ has a pole at $\infty$ of multiplicity $k$, then
$\theta_i(1/z)$ has a pole at $0$ of multiplicity $k$.
Since $\theta_{n-i}^{\vee}$ is holomorphic at $0$,
the identity
\[
\theta_i(1/z)
=
\theta_{n-i}^{\vee}(z)\theta_n(1/z)^p
\]
implies that $\theta_n(1/z)^p$ has a pole at $0$ of multiplicity at least $k$.
Hence $\theta_n^p$ has a pole at $\infty$ of multiplicity at least $k$.
This completes the proof.
\end{proof}
	
We now establish the connections between $\mathbf{\Theta}_n$-inner
functions and the corresponding notions of $\Gamma_n$-inner,
$\mathbb{E}$-inner, and $\mathbf{\Theta}_{n+1}$-inner functions.

\begin{prop}\label{Prop 7}
Let $\Theta=(\theta_1,\ldots,\theta_n)$ be a
$\mathbf{\Theta}_n$-inner function. Then
$(\theta_1,\ldots,\theta_{n-1},\theta_n^p)$ is a
$\Gamma_n$-inner function.
\end{prop}

\begin{proof}
Since $\Theta$ is a $\mathbf{\Theta}_n$-inner function,
$
\Theta(\mathbb D)\subseteq \overline{\mathbf{\Theta}}_n
$
and
$
\Theta(z)\in b\mathbf{\Theta}_n~\text{for a.e. }z\in\mathbb T.$
By Proposition \ref{Prop 3},
\[
(\theta_1(z),\ldots,\theta_{n-1}(z),\theta_n^p(z))
\in b\Gamma_n
\quad\text{for a.e. } z\in\mathbb T.
\]
Moreover, Lemma \ref{Lem 1 Bar} implies that
\[
(\theta_1(z),\ldots,\theta_{n-1}(z),\theta_n^p(z))
\in \Gamma_n
\quad\text{for all } z\in\mathbb D.
\]
Hence $(\theta_1,\ldots,\theta_{n-1},\theta_n^p)$ is a
$\Gamma_n$-inner function.
\end{proof}

\begin{prop}\label{Prop 8}
Let $\Theta=(\theta_1,\ldots,\theta_n)$ be a
$\mathbf{\Theta}_n$-inner function. Then, for every
$\alpha\in\mathbb T$,
\[
(\alpha^m+\theta_1,\,
\alpha^m\theta_1+\theta_2,\,
\ldots,\,
\alpha^m\theta_{n-1}+\theta_n^p,\,
\alpha^{m/p}\theta_n)
\]
is a $\mathbf{\Theta}_{n+1}$-inner function.
\end{prop}

\begin{proof}
Since $\Theta$ is a $\mathbf{\Theta}_n$-inner function,
\[
\Theta(\mathbb D)\subseteq \overline{\mathbf{\Theta}}_n
\quad\text{and}\quad
\Theta(z)\in b\mathbf{\Theta}_n
\text{ for a.e. }z\in\mathbb T.
\]
By Lemma \ref{Lem 2 Bar},
\[
(\alpha^m+\theta_1(z),\,
\alpha^m\theta_1(z)+\theta_2(z),\,
\ldots,\,
\alpha^m\theta_{n-1}(z)+\theta_n^p(z),\,
\alpha^{m/p}\theta_n(z))
\in \overline{\mathbf{\Theta}}_{n+1}
\]
for every $\alpha\in\mathbb T$ and every $z\in\mathbb D$.
Furthermore, Proposition \ref{Prop 4} yields
\[
(\alpha^m+\theta_1(z),\,
\alpha^m\theta_1(z)+\theta_2(z),\,
\ldots,\,
\alpha^m\theta_{n-1}(z)+\theta_n^p(z),\,
\alpha^{m/p}\theta_n(z))
\in b\mathbf{\Theta}_{n+1}
\]
for a.e. $z\in\mathbb T$.
Hence the above $(n+1)$-tuple is a
$\mathbf{\Theta}_{n+1}$-inner function.
\end{proof}
The next result relates $\mathbf{\Theta}_n$-inner functions to
$\mathbb{E}$-inner functions.
\begin{prop}\label{Prop 9}
Let $
k(i)=\binom{n-1}{i}+\binom{n-1}{n-i},$
and let
$\Theta=(\theta_1,\ldots,\theta_n):
\mathbb D\to\mathbb C^n$
be a $\mathbf{\Theta}_n$-inner function.
Then, for each $1\le i\le n-1$,
$
\left(
\frac{\theta_i}{k(i)},
\frac{\theta_{n-i}}{k(i)},
\theta_n^p
\right)$
is an $\mathbb E$-inner function.
\end{prop}

\begin{proof}
Since $\Theta$ is a $\mathbf{\Theta}_n$-inner function,
$
\Theta(z)\in\overline{\mathbf{\Theta}}_n
\quad\text{for all }z\in\mathbb D,
$
and
$
\Theta(z)\in b\mathbf{\Theta}_n
\quad\text{for a.e. }z\in\mathbb T.$
By Lemma \ref{Lem 3},
\[
\left(
\frac{\theta_i(z)}{k(i)},
\frac{\theta_{n-i}(z)}{k(i)},
\theta_n^p(z)
\right)
\in\overline{\mathbb E}~
\text{for every }~z\in\mathbb D~\text{ and}~1\le i\le n-1.\]
Moreover, Proposition \ref{Prop 5} implies that
\[
\left(
\frac{\theta_i(z)}{k(i)},
\frac{\theta_{n-i}(z)}{k(i)},
\theta_n^p(z)
\right)
\in b\mathbb E~\text{
for a.e. }~z\in\mathbb T.\]
Therefore,
$
\left(
\frac{\theta_i}{k(i)},
\frac{\theta_{n-i}}{k(i)},
\theta_n^p
\right)$
is an $\mathbb E$-inner function for each
$1\le i\le n-1$.
\end{proof}	
	
\subsection{Description of $\mathbf{\Theta}_n$-Inner Functions}

We now derive an explicit description of rational
$\mathbf{\Theta}_n$-inner functions. We begin by recalling the well-known
description of rational inner functions on the unit disc. Let $h$ be a rational function on $\mathbb D$ whose poles lie outside
$\overline{\mathbb D}$. Then $h$ is a rational inner function of degree
$n$ (equivalently, a finite Blaschke product with $n$ factors; see
\cite[Corollary 6.10]{JAZLNY 1}) if and only if there exists a polynomial
$D$ of degree at most $n$, having no zeros in $\overline{\mathbb D}$, such
that
\begin{equation}\label{Description of inner function}
\begin{aligned}
h(z)=\frac{D^{\sim n}(z)}{D(z)}.
\end{aligned}
\end{equation}
This representation provides a complete description of rational inner
functions on $\mathbb D$. We denote the degree of $h$ by $\deg(h)$. To obtain an analogous description for rational
$\mathbf{\Theta}_n$-inner functions, we first introduce the associated
model space.

Let $H^2$ denote the Hardy space on the unit circle $\mathbb T$. Define
\begin{equation}\label{H^2 bar}
\overline{H}^2
=
\{f\in L^2:\overline{f}\in H^2\},
\end{equation}
and
\begin{equation}\label{H^2 minus}
H^2_-=
\{f\in L^2:\widehat{f}(m)=0,\ \forall\,m\ge0\}.
\end{equation}
It is well known that
$
\overline{H}^2=\id_{\mathbb D}H^2_-,$
where $\id_{\mathbb D}(z)=z$ is the identity function on $\mathbb D$.

Now let
$\Theta=(\theta_1,\ldots,\theta_n)$
be a rational $\mathbf{\Theta}_n$-inner function and suppose that
$\deg(\theta_n)=k$. Since $\theta_n$ is an inner function,
$\theta_n^p$ is also an inner function of degree $pk$.
Moreover,
\[
(\theta_1(z),\ldots,\theta_n(z))
\in b\mathbf{\Theta}_n
\quad\text{for a.e. }z\in\mathbb T,
\]
and hence each $\theta_i$ belongs to $H^\infty\subset H^2$.
By Lemma \ref{Lem 5},
\[
\theta_i(z)
=
\overline{\theta_{n-i}(z)}\,\theta_n^p(z)
\quad\text{for a.e. }z\in\mathbb T.
\]
Therefore,
$
\overline{\theta_{n-i}}\theta_n^p\in H^2,$
which implies that
$
\theta_{n-i}\overline{\theta_n^p}\in\overline{H}^2.$
Since $|\theta_n|=1$ almost everywhere on $\mathbb T$, we obtain
\[
\theta_{n-i}\in\theta_n^p\overline{H}^2
=(\id_{\mathbb D}\theta_n^p)H^2_-.
\]
Consequently,
\[
\theta_{n-i}
\in
H^2
\cap
(\id_{\mathbb D}\theta_n^p)H^2_-.
\]
Since $\id_{\mathbb D}\theta_n^p$ is an inner function of degree
$pk+1$, its associated model space is
\begin{equation}\label{Model Space}
\begin{aligned}
H^2\cap(\id_{\mathbb D}\theta_n^p)H^2_-
=
H^2\ominus(\id_{\mathbb D}\theta_n^p)H^2
=: \mathcal K_{\Theta}.
\end{aligned}
\end{equation}
(see, for example, \cite[Definition 3.1.1]{Nikolski}.)
Since $\deg(\id_{\mathbb D}\theta_n^p)=pk+1$, it follows that
\[
\dim\mathcal K_{\Theta}=pk+1.
\]
Finally, for a rational $\mathbf{\Theta}_n$-inner function
$\Theta=(\theta_1,\ldots,\theta_n)$ with
$\deg(\theta_n)=k$, we define
\[
\deg(\Theta):=p\,\deg(\theta_n)=pk.
\]
The remainder of this section is devoted to deriving a complete
description of rational $\mathbf{\Theta}_n$-inner functions, analogous
to the representation \eqref{Description of inner function} for rational
inner functions on $\mathbb D$. This characterization constitutes one
of the main results of the paper.	

\begin{thm}\label{Description Theta_n Inner Function}
Let $
\Theta=(\theta_1,\ldots,\theta_n):\mathbb D\to\overline{\mathbf{\Theta}}_n$
be a rational $\mathbf{\Theta}_n$-inner function of degree $pk$.
Then there exist polynomials $E_1,\ldots,E_{n-1}$ and $D$ satisfying the following:

\begin{enumerate}
\item $\deg(E_i)\le pk$ for $1\le i\le n-1$, and $\deg(D)\le k$;

\item $D(z)\neq0$ for every $z\in\overline{\mathbb D}$;

\item
$
\theta_i=\frac{E_i}{D^p},
1\le i\le n-1,$
and $
\theta_n=\frac{D^{\sim k}}{D}$
on $\overline{\mathbb D}$;

\item $
|E_i(z)|\le k(i)\,|D(z)|^p,
z\in\overline{\mathbb D},\;
1\le i\le n-1;$

\item $
E_i=E_{n-i}^{\sim pk},
1\le i\le n-1.$
\end{enumerate}
Moreover, another collection of polynomials
$E_1',\ldots,E_{n-1}'$ and $D'$
satisfies {\rm(1)--(5)} if and only if there exists
$t\in\mathbb R\setminus\{0\}$ such that
\[
D'=tD
\quad\text{and}\quad
E_i'=t^pE_i,
\quad
1\le i\le n-1.
\]
\end{thm}
\begin{proof}
Let $\Theta=(\theta_1,\dots,\theta_n)$ be a rational
$\mathbf{\Theta}_n$-inner function of degree $pk$. By
Lemma~\ref{Lem 4}, $\theta_n$ is an inner function of degree $k$.
Hence, by the representation \eqref{Description of inner function},
there exists a polynomial $D$ of degree at most $k$, with
$D(z)\neq0$ for every $z\in\overline{\mathbb D}$, such that $
\theta_n=\frac{D^{\sim k}}{D}.$
Consequently,
\begin{equation}\label{theta^p_n}
\theta_n^p=\frac{(D^{\sim k})^p}{D^p}.
\end{equation}

Since $\deg(\theta_n^p)=pk$, Proposition~\ref{Prop 6}(2) implies that
every pole of $\theta_i$ is also a pole of $\theta_n^p$, with
multiplicity no larger than that of $\theta_n^p$. Therefore, for each
$1\le i\le n-1$, there exists a polynomial $E_i$ of degree at most
$pk$ such that
$
\theta_i=\frac{E_i}{D^p}.$
Moreover, $\Theta(\mathbb D)\subseteq\overline{\mathbf{\Theta}}_n$,
and hence estimate \eqref{Estimate theta_i} yields $
|\theta_i(z)|\le k(i), z\in\overline{\mathbb D},$
which is equivalent to
\[
|E_i(z)|\le k(i)|D(z)|^p,
\quad
z\in\overline{\mathbb D}.
\]
This proves {\rm(1)--(4)}.

To prove {\rm(5)}, Lemma~\ref{Lem 5} gives
\begin{equation}\label{Des 1}
\theta_i(z)=\theta_{n-i}^{\vee}(1/z)\theta_n^p(z),
\quad z\in\mathbb T.
\end{equation}
Using \eqref{theta^p_n}, we obtain
\[
\frac{E_i(z)}{D(z)^p}
=
\frac{\overline{E_{n-i}(1/\bar z)}}{\overline{D(1/\bar z)^p}}
\cdot
\frac{z^{pk}\overline{D(1/\bar z)^p}}{D(z)^p},
\]
which simplifies to
\begin{equation}\label{Des 2}
E_i(z)
=
z^{pk}\overline{E_{n-i}(1/\bar z)}
=
E_{n-i}^{\sim pk}(z),
\quad z\in\mathbb T.
\end{equation}
Since both sides are polynomials, the identity theorem implies that
\[
E_i=E_{n-i}^{\sim pk},
\quad
1\le i\le n-1,
\]
proving {\rm(5)}.

For uniqueness, suppose another collection
$E_1',\dots,E_{n-1}',D'$ satisfies {\rm(1)--(5)}. Then
\[
\frac{E_i}{D^p}
=
\frac{E_i'}{D'^p},
\quad
\frac{D^{\sim k}}{D}
=
\frac{D'^{\,\sim k}}{D'}.
\]
Since both $D$ and $D'$ are zero-free on
$\overline{\mathbb D}$ and represent the same rational inner function,
they have the same zeros in
$\mathbb C\setminus\overline{\mathbb D}$, counted with multiplicity.
Hence there exists a constant
$t\in\mathbb C\setminus\{0\}$ such that $
D'=tD.$
Substituting into the second identity gives
\[
\frac{D^{\sim k}}{D}
=
\frac{\bar t\,D^{\sim k}}{tD},
\]
and therefore $\bar t=t$. Thus
$t\in\mathbb R\setminus\{0\}$.
Finally,
\[
\frac{E_i}{D^p}
=
\frac{E_i'}{(tD)^p},
\]
which implies
\[
E_i'=t^pE_i,
\quad
1\le i\le n-1.
\]

If $D'=tD$ and $E_i'=t^pE_i$ for some
$t\in\mathbb R\setminus\{0\}$, then it is immediate that
{\rm(1)--(5)} remain valid. This completes the proof.
\end{proof}

It is important to note that Theorem~\ref{Description Theta_n Inner Function}
provides an explicit representation of the components of a rational
$\mathbf{\Theta}_n$-inner function. However, the conditions
{\rm(1)--(5)} in Theorem~\ref{Description Theta_n Inner Function}
are necessary but, in general, not sufficient for an $n$-tuple
$\Theta=(\theta_1,\ldots,\theta_n)$ to be a rational
$\mathbf{\Theta}_n$-inner function.

Indeed, suppose there exist polynomials
$E_1,\ldots,E_{n-1}$ and $D$ satisfying
{\rm(1)--(5)} and define
\[
\theta_i=\frac{E_i}{D^p},
\quad
1\le i\le n-1,
\quad
\theta_n=\frac{D^{\sim k}}{D}.
\]
Then, by construction,
\[
|\theta_n(z)|=1,
\quad
\theta_i(z)=\overline{\theta_{n-i}(z)}\,\theta_n^p(z),
\quad
z\in\mathbb T.
\]

To conclude that $\Theta$ is a rational
$\mathbf{\Theta}_n$-inner function, it remains to verify that
$\Theta(\mathbb D)\subseteq\overline{\mathbf{\Theta}}_n$.
By Theorem~\ref{Bar Theta_n Charaterization}, this is equivalent to $
\sup_{|w|\le1}|f(w,\Theta(z))|\le1, z\in\mathbb D,$
where $f$ is the rational function defined in
\eqref{f}. Explicitly, this amounts to showing that
\begin{equation}\label{Necessary}
\begin{aligned}
\sup_{|w|\le1}
\left|
\frac{
n(-1)^n\theta_n^p(z)w^{\,n-1}
+(n-1)(-1)^{n-1}\theta_{n-1}(z)w^{\,n-2}
+\cdots-\theta_1(z)}
{
n-(n-1)\theta_1(z)w+\cdots+(-1)^{n-1}\theta_{n-1}(z)w^{\,n-1}}
\right|
\le1
\end{aligned}
\end{equation}
for every $z\in\mathbb D$. Now suppose $z\in\mathbb T$. If
\[
n-(n-1)\theta_1(z)w+\cdots
+(-1)^{n-1}\theta_{n-1}(z)w^{\,n-1}\neq0,
\quad w\in\mathbb T,
\]
then Corollary~\ref{Characterization of bTheta_n} yields
\[
\left|
\frac{
n(-1)^n\theta_n^p(z)w^{\,n-1}
+(n-1)(-1)^{n-1}\theta_{n-1}(z)w^{\,n-2}
+\cdots-\theta_1(z)}
{
n-(n-1)\theta_1(z)w+\cdots+(-1)^{n-1}\theta_{n-1}(z)w^{\,n-1}}
\right|
=1.
\]

Therefore, a necessary condition for \eqref{Necessary} is that
\begin{equation}\label{Necessary1}
n-(n-1)\theta_1(z)w+\cdots
+(-1)^{n-1}\theta_{n-1}(z)w^{\,n-1}\neq0,
\quad
w\in\mathbb D,\ z\in\mathbb T.
\end{equation}
By the characterization of $\Gamma_{n-1}$ (see
\cite[Theorem~2.5]{Biswas 2}), condition
\eqref{Necessary1} is equivalent to
\begin{equation}\label{Necessary2}
(\gamma_1\theta_1(z),\ldots,
\gamma_{n-1}\theta_{n-1}(z))
\in\Gamma_{n-1},
\quad
z\in\mathbb T,
\end{equation}
where $\gamma_i=\frac{n-i}{n}$ for
$1\le i\le n-1$.

Conversely, by the same characterization of $\Gamma_{n-1}$,
condition \eqref{Necessary2} guarantees that the denominator of
$f(w,\Theta(z))$ has no zeros in $\overline{\mathbb D}$.
Together with the identities
\[
|\theta_n(z)|=1,
\quad
\theta_i(z)=\overline{\theta_{n-i}(z)}\,\theta_n^p(z),
\]
this implies \eqref{Necessary}. We therefore obtain the following
partial converse to
Theorem~\ref{Description Theta_n Inner Function}.

\begin{thm}\label{Partial Converse}
Let $E_1,\ldots,E_{n-1}$ and $D$ be polynomials satisfying conditions {\rm(1), (2),} and {\rm(5)} of Theorem \ref{Description Theta_n Inner Function}. Define
\[
\theta_i=\frac{E_i}{D^p}, \quad 1\le i\le n-1,
\]
and
\[
\theta_n=\frac{D^{\sim k}}{D}.
\]
Suppose that $
(\gamma_1\theta_1(z),\ldots,\gamma_{n-1}\theta_{n-1}(z))
\in\Gamma_{n-1}~\text{for every } z\in\mathbb T,$
where $\gamma_i=\frac{n-i}{n}$ for $1\le i\le n-1$. Then $
\Theta=(\theta_1,\ldots,\theta_n)$
is a rational $\mathbf{\Theta}_n$-inner function of degree at most $pk$.
\end{thm}
	
	\begin{proof}
		Suppose $\theta_i = \frac{E_i}{D^p}$ for $1 \le i \le n-1$ and $\theta_n = \frac{D^{\sim k}}{D}$ on $\overline{\mathbb{D}}$. To show that $\Theta := (\theta_1, \dots, \theta_n)$ is a rational $\mathbf{\Theta}_n$-inner function. In order to prove that we need to verify the following:
		\begin{enumerate}
			\item[(i)] $\theta_n$ is inner,
			
			\item[(ii)] $\Theta(z) \in b\mathbf{\Theta}_n$ for a.e. $z \in \mathbb{T}$,
			
			\item[(iii)] $\Theta : \mathbb{D} \to \overline{\mathbf{\Theta}}_n$; i.e., $\Theta$ maps $\mathbb{D}$ into $\overline{\mathbf{\Theta}}_n$.
		\end{enumerate}
	First, we prove (i). Since $D$ satisfies condition {\rm(2)}, the
representation $
\theta_n=\frac{D^{\sim k}}{D}$
is precisely the standard representation of a rational inner function.
Hence $\theta_n$ is inner. Moreover, if $D$ and $D^{\sim k}$ have common
zeros on $\mathbb T$, these common factors cancel. Consequently, $
\deg(\theta_n)\le k.$	

By condition {\rm(5)} of
Theorem~\ref{Description Theta_n Inner Function},
$
E_i=E_{n-i}^{\sim pk},
1\le i\le n-1.$
Hence, by \eqref{Des 1} and \eqref{Des 2},
\[
\theta_i(z)
=
\overline{\theta_{n-i}(z)}
\theta_n^p(z),
\quad
z\in\mathbb T.
\]
Since $
(\gamma_1\theta_1(z),\ldots,\gamma_{n-1}\theta_{n-1}(z))
\in\Gamma_{n-1}$
for every $z\in\mathbb T$, it follows from
\cite[Theorem~2.5]{Biswas 2} that
$
\Theta(z)\in b\mathbf{\Theta}_n$
for almost every $z\in\mathbb T$.
Thus (ii) holds.		

To prove (iii), by Theorem \ref{Bar Theta_n Charaterization}, it is
enough to show that
\begin{equation}\label{Con Des 2}
\sup_{|w|\le1}|f(w,\Theta(z))|
\le1
\quad\text{for every }z\in\overline{\mathbb D}.
\end{equation}
Since $\Theta(z)\in b\mathbf{\Theta}_n$ for almost every
$z\in\mathbb T$, Theorem
\ref{Bar Theta_n Charaterization} yields
$
|f(w,\Theta(z))|
\le1$
for every fixed $w\in\overline{\mathbb D}$ and for almost every
$z\in\mathbb T$.
Now fix $w\in\overline{\mathbb D}$. Since
$
(\gamma_1\theta_1(z),\ldots,\gamma_{n-1}\theta_{n-1}(z))
\in\Gamma_{n-1}$
for every $z\in\mathbb T$, the denominator of
$f(w,\Theta(z))$ is nonzero for every
$w\in\overline{\mathbb D}$. Hence the function
\[
z\longmapsto f(w,\Theta(z))
\]
is holomorphic on $\mathbb D$. Therefore, by the maximum modulus
principle,
\[
|f(w,\Theta(z))|
\le1
\quad
\text{for every }z\in\mathbb D.
\]
Since this holds for every
$w\in\overline{\mathbb D}$,
Theorem \ref{Bar Theta_n Charaterization}
implies that
$
\Theta(\mathbb D)\subseteq\overline{\mathbf{\Theta}}_n.$
Consequently,
$\Theta$ is a rational map from
$\mathbb D$ into
$\overline{\mathbf{\Theta}}_n$
whose radial boundary values belong to
$b\mathbf{\Theta}_n$
almost everywhere on $\mathbb T$.
Hence $\Theta$ is a rational
$\mathbf{\Theta}_n$-inner function.
This completes the proof.
	\end{proof}

	\begin{rem}\label{Rem 3}
Let $\Theta=(\theta_1,\dots,\theta_n)$ be a rational
$\mathbf{\Theta}_n$-inner function. Then, by
Lemma~\ref{Lem 4},
\[
\theta_i(z)=\overline{\theta_{n-i}(z)}\,\theta_n^p(z)
\quad\text{and}\quad
|\theta_n(z)|=1
\]
for almost every $z\in\mathbb T$. Consequently,
$
|\theta_i(z)|=|\theta_{n-i}(z)|$
for almost every $z\in\mathbb T$ and for every
$1\le i\le n-1$.
\end{rem}

	\subsection{Interpolation by Rational $\mathbf{\Theta}_n$-Inner Functions}

The \textit{classical Nevanlinna-Pick interpolation problem} on the unit disc asserts that if there exists a holomorphic function
$f:\mathbb D\to\mathbb D$
satisfying
$f(z_i)=w_i$
for $1\le i\le m$, then there exists a finite Blaschke product $B$ of degree at most $m$ such that
\[
B(z_i)=w_i,\quad 1\le i\le m.
\]
This shows that every solvable interpolation problem on $\mathbb D$ admits a rational inner interpolant.
A holomorphic map
$\widetilde{F}:\mathbb D\to M_n(\mathbb C)$
is called a \emph{rational inner matrix function} if its entries are rational functions with poles outside $\overline{\mathbb D}$ and
$\widetilde{F}(\zeta)$ is unitary for almost every $\zeta\in\mathbb T$.
Rational inner matrix functions play the analogous role in the matrix-valued Nevanlinna-Pick interpolation problem.
The following theorem is the corresponding analogue for interpolation in $M_n(\mathbb C)$.
\begin{thm}[\cite{McCarthy 1}]\label{Interpolation in M_n(C)}
Let $z_1,\ldots,z_m$ be distinct points in $\mathbb D$, and let
$W_1,\ldots,W_m\in M_n(\mathbb C)$.
Suppose there exists a holomorphic function
$F:\mathbb D\to M_n(\mathbb C)$
such that
\[
\|F(z)\|\le1,\quad z\in\mathbb D,
\]
and
\[
F(z_i)=W_i,\quad 1\le i\le m.
\]
Then there exists a holomorphic function
$G:\mathbb D\to M_n(\mathbb C)$,
whose entries are rational functions with poles outside
$\overline{\mathbb D}$, such that
$
G(\zeta)\ \text{is unitary for every } \zeta\in\mathbb T,$
and
$
G(z_i)=W_i,~ 1\le i\le m.$
\end{thm}

Consider the map $J_n : \mathbf{\Theta}_n \to \Omega_n$ defined by
\begin{equation}\label{J_n}
\begin{aligned}
J_n(\theta_1,\dots,\theta_n)
=
\begin{pmatrix}
0&0&\cdots&0&(-1)^{n-1}\theta_n^p\\
1&0&\cdots&0&(-1)^{n-2}\theta_{n-1}\\
\vdots&\vdots&\ddots&\vdots&\vdots\\
0&0&\cdots&0&-\theta_2\\
0&0&\cdots&1&\theta_1
\end{pmatrix},
\end{aligned}
\end{equation}
which is the companion matrix associated with the polynomial $P$.
Several approaches have been developed to solve the classical
Nevanlinna--Pick interpolation problem.
Among them,
Bercovici, Foias, and Tannenbaum~\cite{Bercovici}
obtained a proof using the Sz.-Nagy--Foias
commutant lifting theorem (see also \cite{Foias}).
Later, Costara~\cite[Theorem~1.1]{costara}
established the following interpolation criterion for the spectral ball.
	\begin{thm}\cite[Theorem~1.1]{costara}\label{Inter}
Let $z_1,\dots,z_m\in\mathbb D$ and
$W_1,\dots,W_m\in M_n(\mathbb C)$.
There exists a bounded analytic function
$F:\mathbb D\to M_n(\mathbb C)$ satisfying
\[
F(z_j)=W_j,\quad 1\le j\le m,
\]
and
\[
\sup_{|z|<1} r(F(z))<1,
\]
if and only if there exist invertible matrices
$X_1,\dots,X_m\in M_n(\mathbb C)$
and an analytic function
$G:\mathbb D\to M_n(\mathbb C)$
such that
$
G(z_j)=X_jW_jX_j^{-1}, 1\le j\le m,$
and
$
\sup_{|z|<1}\|G(z)\|<1.$
\end{thm}
In this subsection, we study the interpolation problem for
$\mathbf{\Theta}_n$. As a first application of
Theorem~\ref{Theta_n Characterization}, we obtain a necessary condition
for the existence of a holomorphic interpolating map into
$\mathbf{\Theta}_n$.

\begin{thm}\label{Interpolation}
Let $z_1,\ldots,z_N\in\mathbb D$ be distinct points, and let
\[
(\theta_{1,j},\ldots,\theta_{n,j})\in\mathbf{\Theta}_n,
\quad
1\le j\le N.
\]
Suppose there exists a holomorphic map
$h:\mathbb D\to\mathbf{\Theta}_n$ satisfying
\[
h(z_j)=
(\theta_{1,j},\ldots,\theta_{n,j}),
\quad
1\le j\le N.
\]
Then, for every fixed $z\in\mathbb D$, the Pick matrix
\[
\left(
\frac{1-\overline{w_j}w_i}
{1-\overline{z_j}z_i}
\right)_{i,j=1}^N
\]
is positive semidefinite, where
\[
w_j=
\frac{
n(-1)^n\theta_{n,j}^{\,p}z^{\,n-1}
+(n-1)(-1)^{n-1}\theta_{n-1,j}z^{\,n-2}
+\cdots-\theta_{1,j}
}
{
n-(n-1)\theta_{1,j}z
+\cdots
+(-1)^{n-1}\theta_{n-1,j}z^{\,n-1}
},
\quad
1\le j\le N.
\]
\end{thm}

	\begin{proof}
Suppose there exists a holomorphic map $
h:\mathbb D\longrightarrow\mathbf{\Theta}_n$
such that
\[
h(z_j)=
(\theta_{1,j},\ldots,\theta_{n,j}),
\quad
1\le j\le N.
\]
Fix $z\in\mathbb D$ and define
$
g:\mathbf{\Theta}_n\longrightarrow\mathbb C
$
by
\[
g(\theta_1,\ldots,\theta_n)
=
\frac{
n(-1)^n\theta_n^p z^{n-1}
+(n-1)(-1)^{n-1}\theta_{n-1}z^{n-2}
+\cdots-\theta_1
}
{
n-(n-1)\theta_1z
+\cdots
+(-1)^{n-1}\theta_{n-1}z^{n-1}
}.
\]
Thus $g$ is obtained by fixing the parameter $z$ in the rational function
$f$ defined in \eqref{f}. Since
$(\theta_1,\ldots,\theta_n)\in\mathbf{\Theta}_n$,
Theorem~\ref{Theta_n Characterization} implies that the denominator is
nonzero on $\overline{\mathbb D}$ and
\[
|g(\theta_1,\ldots,\theta_n)|
=
|f(z;\theta_1,\ldots,\theta_n)|
<1.
\]
Hence $g$ is holomorphic on $\mathbf{\Theta}_n$ and maps
$\mathbf{\Theta}_n$ into $\mathbb D$.
Therefore,
\[
F_z:=g\circ h:\mathbb D\longrightarrow\mathbb D
\]
is holomorphic. Moreover,
\[
F_z(z_j)
=
g(h(z_j))
=
g(\theta_{1,j},\ldots,\theta_{n,j})
=
w_j,
\quad
1\le j\le N.
\]
By the classical Nevanlinna--Pick theorem, the Pick matrix
\[
\left(
\frac{1-\overline{w_j}w_i}
{1-\overline{z_j}z_i}
\right)_{i,j=1}^N
\]
is positive semidefinite. This completes the proof.
\end{proof}
The following theorem is analogous to Theorem \ref{Inter}.
\begin{thm}\label{Interpolation in Theta_n}
		Let $z_1, \dots, z_N$ be distinct points in $\mathbb{D}$ and $(\theta_{1, j}, \dots, \theta_{n, j}) \in \mathbf{\Theta}_n$ for $1 \le j \le N$. If there exists an holomorphic function $f : \mathbb{D} \to \mathbf{\Theta}_n$ such that $f(z_j) = (\theta_{1, j}, \dots, \theta_{n, j})$ for $1 \le j \le N$, then there exists a rational $\mathbf{\Theta}_n$-inner function $g : \mathbb{D} \to \mathbf{\Theta}_n$ such that $g(z_j) = (\theta_{1, j}, \dots, \theta_{n, j})$ for $1 \le j \le N$.
	\end{thm}
	
\begin{proof}
Suppose there exists a holomorphic map
$f:\mathbb D\to\mathbf{\Theta}_n$
satisfying
\[
f(z_j)=
(\theta_{1,j},\ldots,\theta_{n,j}),
\quad
1\le j\le N.
\]
Define
\[
F=J_n\circ f:\mathbb D\to\Omega_n.
\]
Then $F$ is bounded and
\[
F(z_j)=J_n(\theta_{1,j},\ldots,\theta_{n,j}),
\quad
1\le j\le N.
\]
For $0<\gamma<1$, we have
\[
\sup_{|z|<1}r(\gamma F(z))
=\gamma\sup_{|z|<1}r(F(z))
<1.
\]
Hence, by Theorem~\ref{Inter}, there exists a holomorphic map
$
G_\gamma:\mathbb D\to M_n(\mathbb C)$
such that
$G_\gamma(z_j)$ is similar to
$\gamma F(z_j)$ for each $j$, and
$
\sup_{|z|<1}\|G_\gamma(z)\|<1.$
Thus $\{G_\gamma\}_{0<\gamma<1}$ is uniformly bounded.
By Montel's theorem, there exist a sequence
$\gamma_m\to1$ and a holomorphic map
$G:\mathbb D\to M_n(\mathbb C)$ such that
$G_{\gamma_m}\to G$
uniformly on compact subsets of $\mathbb D$.
Consequently,
$\|G(z)\|\le1$ for all $z\in\mathbb D$.
Since the spectrum depends continuously on matrix entries,
\[
\sigma(G(z_j))
=
\sigma(F(z_j)),
\quad
1\le j\le N.
\]
Applying Theorem~\ref{Interpolation in M_n(C)},
there exists a rational inner matrix function
$
\widetilde G:\mathbb D\to M_n(\mathbb C)$
such that
\[
\widetilde G(z_j)=G(z_j),
\quad
1\le j\le N.
\]
Define
$
g=\Pi_n\circ\widetilde G.$
Since $\Pi_n$ is polynomial,
$g$ is a rational holomorphic map.
Moreover, the boundary values of
$\widetilde G$
are unitary almost everywhere on $\mathbb T$.
Hence
$
\Pi_n(\widetilde G(e^{it}))
\in b\mathbf{\Theta}_n$
for almost every $e^{it}\in\mathbb T$,
which shows that
$g$ is a rational
$\mathbf{\Theta}_n$-inner function.
Finally,
\[
\begin{aligned}
g(z_j)
&=\Pi_n(\widetilde G(z_j))
 =\Pi_n(G(z_j))
 =\theta(\sigma(G(z_j)))  \\
&=\theta(\sigma(F(z_j)))
 =\Pi_n(F(z_j))
 =f(z_j)
 =(\theta_{1,j},\ldots,\theta_{n,j}),
\end{aligned}
\]
for every $1\le j\le N$.
This completes the proof.
\end{proof}	
	
We now present the main result of this subsection, which provides an explicit description of the rational $\mathbf{\Theta}_n$-inner interpolants corresponding to the interpolation data in Theorem~\ref{Interpolation in Theta_n}.
	
	\begin{thm}\label{Interpolation Main Result}
		Let
\(z_1,\ldots,z_N\)
be distinct points in \(\mathbb D\), and let
\[
(\theta_{1,j},\ldots,\theta_{n,j})
\in\mathbf{\Theta}_n,
\quad
1\le j\le N.
\]
Suppose there exists a holomorphic map
\[
f:\mathbb D\to\mathbf{\Theta}_n
\]
such that
\[
f(z_j)=
(\theta_{1,j},\ldots,\theta_{n,j}),
\quad
1\le j\le N.
\]
Then there exists a rational
\(\mathbf{\Theta}_n\)-inner function
\(g:\mathbb D\to\mathbf{\Theta}_n\)
interpolating the given data.
Moreover, \(g\) admits the following explicit representation.		\begin{enumerate}
			\item Suppose that $n=2l+1$, where $l\in\mathbb N$.
Then there exist
\[
k\in\mathbb N,\quad
\xi\in\mathbb T,\quad
\alpha_1,\ldots,\alpha_k\in\mathbb D,
\]
such that, for each $1\le j\le l$, there exist
\[
t_j,q_j\in\mathbb N\cup\{0\},\qquad
r_j\in\mathbb R,
\]
\[
\eta_j,\delta_j,\xi_{j,1},\ldots,\xi_{j,t_j}\in\mathbb T,
\qquad
\beta_{j,1},\ldots,\beta_{j,q_j}\in\mathbb D,
\]
satisfying
			\begin{equation*}
				\begin{aligned}
					t_j + q_j &= pk, \,\, \eta_j\delta_j\left(\prod_{i=1}^{t_j}\xi_{j, i}\right) = \xi \,\, \text{for} \,\, 1 \le j \le l.
				\end{aligned}
			\end{equation*}
			Define
\[
\theta_n(z)
=
\xi^{1/p}\prod_{i=1}^k
\frac{z-\alpha_i}{1-\overline{\alpha_i}z},
\]
and, for each $1\le j\le l$, define
			\begin{equation}\label{theta_j, theta_{n-j} Description}
				\begin{aligned}
					&\theta_j(z) =
					r_j\eta_j\frac{\left(\displaystyle \prod_{i=1}^{t_j} (\xi_{j, i} + z)\right)\left(\displaystyle \prod_{i=1}^{q_j} f_{j, i}(z)\right)}{\left(\displaystyle \prod_{i=1}^{k} (\overline{\alpha}_iz - 1)^p\right)}
					\,\,\,\, \text{and} \,\,\,\, \theta_{n-j}(z) =
					r_j\delta_j\frac{\left(\displaystyle \prod_{i=1}^{t_j} (\xi_{j, i} + z)\right)\left(\displaystyle \prod_{i=1}^{q_j} g_{j, i}(z)\right)}{\left(\displaystyle \prod_{i=1}^{k} (\overline{\alpha}_iz - 1)^p\right)}
				\end{aligned}
			\end{equation}
			where, for $1 \le i \le q_j$, either $f_{j, i}(z) = \overline{\beta}_{j, i}z - 1$ and $g_{j, i}(z) = \beta_{j, i} - z$, or $f_{j, i}(z) = \beta_{j, i} - z$ and $g_{j, i}(z) = \overline{\beta}_{j, i}z - 1$.  The rational map
\[
g(z)=\bigl(\theta_1(z),\ldots,\theta_n(z)\bigr)
\]
is the required $\mathbf{\Theta}_n$-inner function interpolating the prescribed data.		
			\item If $n = 2l$ for $l \in \mathbb{N}$, then, for $1 \le j \le l-1$, the functions $\theta_j$ are defined as in
\eqref{theta_j, theta_{n-j} Description}, while $\theta_l$ is given by
			\begin{equation}\label{theta_l Description}
				\begin{aligned}
					\theta_l(z) &= r_l\eta_l \frac{\left(\displaystyle\prod_{i=1}^{t_l}(\xi_{l, i} + z)\right)\left(\displaystyle\prod_{i=1}^{q_l}(\beta_{l, i} - z)(\overline{\beta}_{l, i}z - 1)\right)}{\displaystyle\prod_{i=1}^{k}(\overline{\alpha}_iz - 1)^p},
				\end{aligned}
			\end{equation}
			where $t_l, q_l \in \mathbb{N} \cup \{0\}$ and $t_l + 2q_l = pk, r_l \in \mathbb{R}, \beta_{l, 1}, \dots, \beta_{l, q_l} \in \mathbb{D}$ and $\eta_l, \xi_{l, 1}, \dots, \xi_{l, t_l} \in \mathbb{T}$ are such that $\eta^2_l\left(\displaystyle\prod_{i=1}^{t_l}\xi_{l, i}\right) = \xi$. The rational map
\[
g(z)=\bigl(\theta_1(z),\ldots,\theta_n(z)\bigr)
\]
is the required $\mathbf{\Theta}_n$-inner function interpolating the prescribed data.
		\end{enumerate}
	\end{thm}
	
	\begin{proof}
		Since there exists a holomorphic map
$
f:\mathbb D\to\mathbf{\Theta}_n$
interpolating the prescribed data, Theorem~\ref{Interpolation in Theta_n}
yields a rational $\mathbf{\Theta}_n$-inner function
$
g:\mathbb D\to\mathbf{\Theta}_n$
such that
$
g(z_j)=f(z_j),1\le j\le N.$
Replacing \(g\) by \(f\), we may therefore assume without loss of
generality that \(f\) itself is a rational
$\mathbf{\Theta}_n$-inner function.
Write $
f=(\theta_1,\ldots,\theta_n).$
Since $f$ is rational $\mathbf{\Theta}_n$-inner then $\theta_n$ is inner. Hence there exist $\xi\in\mathbb T$ and
$\alpha_1,\ldots,\alpha_k\in\mathbb D$ such that
\[
\theta_n(z)
=
\xi^{1/p}
\prod_{i=1}^{k}
\frac{z-\alpha_i}{1-\overline{\alpha_i}z}.
\]
Consequently,
\[
\theta_n^p(z)
=
\xi
\prod_{i=1}^{k}
\left(
\frac{z-\alpha_i}{1-\overline{\alpha_i}z}
\right)^p.
\]
Every element of the model space $\mathcal K_{\Theta}$ is of the form
$P/Q$, where $P$ is a polynomial of degree at most $pk$ and
\[
Q(z)=\prod_{i=1}^{k}(1-\overline{\alpha_i}z)^p.
\]
Since $\theta_j\in\mathcal K_{\Theta}$ for $1\le j\le n-1$, there exists a
polynomial $P_j$ of degree at most $pk$ such that
\[
\theta_j(z)=\frac{P_j(z)}{Q(z)}.
\]
Since
\[
\theta_j(z)=\overline{\theta_{n-j}(z)}\,\theta_n^p(z)
\quad \text{a.e.}~ z\in\mathbb T,
\]
it follows that
\begin{equation}\label{Inter 1}
\frac{\theta_j(z)\theta_{n-j}(z)}
{\theta_n^p(z)}
=
|\theta_{n-j}(z)|^2
\quad
 \text{a.e.}~ z\in\mathbb T.
\end{equation}
Substituting the above representations into \eqref{Inter 1}, we obtain
\begin{equation}\label{Inter 2}
\frac{\theta_j(z)\theta_{n-j}(z)}
{\theta_n^p(z)}
=
\frac{P_j(z)P_{n-j}(z)}
{\displaystyle
\xi\prod_{i=1}^{k}
(z-\alpha_i)^p
(1-\overline{\alpha_i}z)^p}.
\end{equation}
Since \eqref{Inter 1} shows that
\[
\frac{\theta_j(z)\theta_{n-j}(z)}{\theta_n^p(z)}
\]
is a non-negative rational function on $\mathbb T$, it follows from
\cite[p.~137]{Duren} that there exist
\[
R_j\ge0,\quad
t_j,q_j\in\mathbb N\cup\{0\},
\]
with $t_j+q_j=pk$,
\[
\xi_{j,1},\ldots,\xi_{j,t_j}\in\mathbb T,
\quad
\beta_{j,1},\ldots,\beta_{j,q_j}\in\mathbb D,
\]
such that
\begin{equation}\label{Inter 3}
\begin{aligned}
\frac{\theta_j(z)\theta_{n-j}(z)}{\theta_n^p(z)}
=
R_j
\frac{
\left(\prod_{i=1}^{t_j}(\xi_{j,i}+z)^2\right)
\left(\prod_{i=1}^{q_j}
(\beta_{j,i}-z)(\overline{\beta}_{j,i}z-1)\right)}
{\left(\prod_{i=1}^{t_j}\xi_{j,i}\right)
\left(\prod_{i=1}^{k}
(z-\alpha_i)^p(1-\overline{\alpha_i}z)^p\right)}.
\end{aligned}
\end{equation}
Now suppose that $n=2l+1$. Then $j\neq n-j$ for
$1\le j\le l$. Since
\[
\theta_j(z)
=
\overline{\theta_{n-j}(z)}\,\theta_n^p(z)
\quad  \text{a.e.}~ z\in\mathbb T,
\]
there exist
$r_j\in\mathbb R$ and
$\eta_j,\delta_j\in\mathbb T$
satisfying
\[
\eta_j\delta_j
\prod_{i=1}^{t_j}\xi_{j,i}
=
\xi,
\]
such that
\begin{equation}\label{Inter 4}
\begin{aligned}
\theta_j(z)
&=
r_j\eta_j
\frac{\left(\prod_{i=1}^{t_j}(\xi_{j,i}+z)\right)
\left(\prod_{i=1}^{q_j}f_{j,i}(z)\right)}
{\prod_{i=1}^{k}(1-\overline{\alpha_i}z)^p},\\
\theta_{n-j}(z)
&=
r_j\delta_j
\frac{\left(\prod_{i=1}^{t_j}(\xi_{j,i}+z)\right)
\left(\prod_{i=1}^{q_j}g_{j,i}(z)\right)}
{\prod_{i=1}^{k}(1-\overline{\alpha_i}z)^p},
\end{aligned}
\end{equation}
where, for each $1\le i\le q_j$, either
\[
f_{j,i}(z)=\overline{\beta}_{j,i}z-1,
\quad
g_{j,i}(z)=\beta_{j,i}-z,
\]
or
\[
f_{j,i}(z)=\beta_{j,i}-z,
\quad
g_{j,i}(z)=\overline{\beta}_{j,i}z-1.
\]
Consequently,
$
g(z)=\bigl(\theta_1(z),\ldots,\theta_n(z)\bigr)$
has the required form and interpolates the prescribed data. This proves
\((1)\).

$(2):$ Now suppose that $n=2l$. For $1\le j\le l-1$, the functions
$\theta_j$ and $\theta_{n-j}$ are given by \eqref{Inter 4}. Since
$\theta_l=\overline{\theta_l}\theta_n^p$ on $\mathbb T$, it follows that
$\theta_l^2/\theta_n^p$ is a non-negative rational function on
$\mathbb T$. Hence, arguing as above, there exist
$t_l,q_l\in\mathbb N\cup\{0\}$,
$r_l\in\mathbb R$,
$\beta_{l,1},\ldots,\beta_{l,q_l}\in\mathbb D$, and
$\eta_l,\xi_{l,1},\ldots,\xi_{l,t_l}\in\mathbb T$
satisfying
\[
t_l+2q_l=pk,
\qquad
\eta_l^2\prod_{i=1}^{t_l}\xi_{l,i}=\xi,
\]
such that
\begin{equation}\label{Inter 5}
\begin{aligned}
\theta_l(z)
=
r_l\eta_l
\frac{\left(\displaystyle\prod_{i=1}^{t_l}(\xi_{l,i}+z)\right)
\left(\displaystyle\prod_{i=1}^{q_l}
(\beta_{l,i}-z)(\overline{\beta}_{l,i}z-1)\right)}
{\displaystyle\prod_{i=1}^{k}(1-\overline{\alpha_i}z)^p}.
\end{aligned}
\end{equation}
Consequently,
$
g(z)=\bigl(\theta_1(z),\ldots,\theta_n(z)\bigr)$
has the required form and interpolates the prescribed data.
	 This completes the proof.
	\end{proof}

	\noindent (D. K. Keshari) \sc{School of Mathematical Sciences, National Institute of Science Education and Research, Bhubaneswar, An OCC of Homi Bhabha National Institute, Jatni, Khurda, 752050, India}\\
	{E-mail address:} {dinesh@niser.ac.in}
	
	\vspace{.5cm}
	
	\noindent (S. Nayak) \sc{School of Mathematical Sciences, National Institute of Science Education and Research, Bhubaneswar, An OCC of Homi Bhabha National Institute, Jatni, Khurda, 752050, India}\\
	{E-mail address:} {suryanarayan.nayak@niser.ac.in}
	
	\vspace{.5cm}
	
	\noindent (A. Pal) \sc{Department of Mathematics, IIT Bhilai, 6th Lane Road, Jevra, Chhattisgarh 491002}\\
	{E-mail address:} {avijit@iitbhilai.ac.in}
	
	\vspace{.5cm}
	
	\noindent (B. Paul) \sc{Department of Mathematics, IIT Bhilai, 6th Lane Road, Jevra, Chhattisgarh 491002}\\
	{E-mail address:} {bhaskarpaul@iitbhilai.ac.in}  

\end{document}